\documentclass{amsart}
\usepackage{amssymb}
\usepackage{tikz}
\usepackage{verbatim}
\usepackage{hyperref}
\usepackage{orcidlink}
\usepackage[english]{babel}
\usepackage[shortlabels]{enumitem}
\usepackage{mathtools}
\usepackage{mathrsfs}
\usepackage{bm}
\usepackage{stmaryrd} 
\usepackage{graphicx}
\usepackage{float}

\newtheorem*{cla}{Claim}
\newtheorem*{cla1}{Claim 1}
\newtheorem*{cla2}{Claim 2}
\newtheorem*{cla3}{Claim 3}
\newtheorem*{cla4}{Claim 4}

\newenvironment{claimproof}{\paragraph{\em Proof of claim.}}{\hfill\scalebox{0.7}{$\blacksquare$}}

\newtheorem{theorem}{Theorem}[section]
\newtheorem{lemma}[theorem]{Lemma}
\newtheorem{corollary}[theorem]{Corollary}
\newtheorem{proposition}[theorem]{Proposition}

\theoremstyle{definition}
\newtheorem{definition}[theorem]{Definition}

\newtheorem{example}[theorem]{Example}

\theoremstyle{remark}

\theoremstyle{remark}

\numberwithin{equation}{section}

\newcommand{\alg}[1]{\mathbf{#1}}
\newcommand{\var}[1]{\mathcal{#1}}
\newcommand{\rel}[1]{\mathbb{#1}}

\newcommand{\pol}{\operatorname{Pol}}
\newcommand{\clo}{\operatorname{Clo}}

\newcommand{\id}{\operatorname{id}}

\DeclareMathAlphabet\mathbfcal{OMS}{cmsy}{b}{n}

\newcommand{\LL}{\mathbfcal{L}}

\newcommand{\mA}{\mathbb{A}}
\newcommand{\mB}{\mathbb{B}}
\newcommand{\mC}{\mathbb{C}}
\newcommand{\mS}{\mathbb{S}}
\newcommand{\mD}{\mathbb{D}}
\newcommand{\mE}{\mathbb{E}}
\newcommand{\mF}{\mathbb{F}}
\newcommand{\mG}{\mathbb{G}}
\newcommand{\mH}{\mathbb{H}}
\newcommand{\mM}{\mathbb{M}}
\newcommand{\mK}{\mathbb{K}}
\newcommand{\mL}{\mathbb{L}}
\newcommand{\mP}{\mathbb{P}}

\newcommand{\mY}{\mathbb{Y}}
\newcommand{\me}{\mathbb{I}}

\newcommand{\Id}{\operatorname{Id}}

\newcommand{\vege}{\hfill$\bigstar$}

\begin{document}

\title[The filter of interpretability types of Hobby-McKenzie varieties]{The filter of interpretability types of Hobby-McKenzie varieties is prime}

\author{Bertalan Bodor \orcidlink{0009-0003-6679-6355}}
\author{Gerg\H o Gyenizse \orcidlink{0000-0001-6864-8995}} 
\author{Mikl\'os Mar\'oti \orcidlink{0000-0002-3326-2512}}
\author{L\'aszl\'o Z\'adori \orcidlink{0000-0002-2383-2294}}

\address{HUN-REN Alfréd Rényi Institute of Mathematics, Budapest, Reáltanoda utca 13-15, HUNGARY 1053}
\email{bodor@renyi.hu}
\address{Bolyai Institute, Univ. of Szeged, Szeged, Aradi V\'{e}rtan\'{u}k tere 1, HUNGARY 6720}
\email{gergogyenizse@gmail.com}    
\email{mmaroti@math.u-szeged.hu} 
\email{zadori@math.u-szeged.hu}

\thanks{The research of authors was supported by Project no TKP2021-NVA-09 financed by the Ministry of Culture and Innovation of Hungary from the National Research, Development and Innovation Fund and the NKFIH grant K138892.\\
Bertalan Bodor has been funded by the European Research Council (Project POCOCOP, ERC Synergy Grant 101071674). Views and opinions expressed are however those of the authors only and do not necessarily reflect those of the European Union or the European Research Council Executive Agency. Neither the European Union nor the granting authority can be held responsible for them.
}


\begin{abstract}
We  obtain new characterizations of the Hobby-McKenzie varieties via compatible reflexive ternary structures. Based on our findings, we prove that in the lattice of interpretability types of varieties, the filter of the interpretability types of Hobby-McKenzie varieties is prime.
\end{abstract}

\maketitle

\section{Introduction}

In~\cite{GT}, Garcia and Taylor initiated a systematic study of Maltsev conditions, certain sets of identities used to describe classes of varieties of similar behavior. They introduced the lattice $\LL$ of interpretability types of varieties and used this lattice  to investigate the interrelationship between classes of varieties determined by Maltsev conditions. It turns out that Maltsev conditions correspond to certain filters (Maltsev filters) of $\LL$. This suggests that better knowledge on the structural properties of the lattice $\LL$ leads to better understanding of  Maltsev conditions.

In~\cite{GT}, one of the main questions investigated  is which classical Maltsev conditions determine a prime filter in $\LL$, and
primeness of several well known Maltsev filters of $\LL$ was decided. On many occasions, Maltsev filters are characterized as the class of varieties whose algebras satisfy some familiar congruence property. The primeness of some Maltsev filters, such as the filter of the types of congruence permutable varieties and the filter of the types of congruence modular varieties, was stated as an open problem.

The ideas presented in~\cite{GT} inspired further research on primeness of classical Maltsev conditions, see e.g.,~\cite{BS,C,KT,O,S,VW}. However, each of the results obtained is only valid in some part of~$\LL$, they are bound to some extra stipulations on the varieties whose interpretability types are considered. The most frequently used conditions are idempotency, linearity, and local finiteness.
Not long ago, in~\cite{GMZ2}, the last  three authors of this paper gave a proof that congruence permutability is indeed a prime Maltsev condition. Thus, this settled one of the open problems of Garcia and Taylor, mentioned above. The other problem on congruence modular varieties is still open.

By extending the results of Hobby and McKenzie in~\cite{HM}, Kearnes and Kiss elaborated a hierarchy of varieties based on certain Maltsev conditions in~\cite{KK}. It is natural to look at the problem of primeness for the Maltsev filters in $\LL$ that were characterized in~\cite{KK}. Recently, in~\cite{BGMZ}, the authors of this article gave a proof that the Taylor filter, one of the filters just mentioned, is prime in $\LL$. In the present paper, we prove the primeness of another Maltsev filter  arising in~\cite{KK} that consists of the interpretability types of the Hobby-McKenzie varieties. We remark, just to compare our result to the open problem on congruence modular varieties, that the Hobby-McKenzie varieties are characterized in~\cite{KK} as the varieties that satisfy a non-trivial congruence identity. We give a precise definition of the  Hobby-McKenzie varieties and the filter of the Hobby-McKenzie interpretability types in the following section, see Definition~\ref{HMv}.
 
In Figure~\ref{fig:hierarch}, we depicted the Maltsev filters of $\LL$ that were used to establish a certain hierarchy of varieties in~\cite{KK}. For the definitions of the Maltsev filters in Figure~\ref{fig:hierarch}, other than the Hobby-McKenzie filter, the interested reader should consult~\cite{KK}. A node is painted black if it symbolizes a prime Maltsev filter, is left empty if it symbolizes a non-prime filter, and is painted gray if its status is unknown at present. The filters in the figure are depicted accordingly to reverse containment. This may be unusual, but we want to keep the direction coming from the interpretability ordering. It is well known that the filters in the figure are pairwise different. We remark that the non-primeness of $n$-permutability for some $n$ was proved by the last three authors of the present paper in~\cite{GMZ1}. For the non-primeness of the other non-prime Maltsev filters in the figure, see Proposition 1.1 in~\cite{BGMZ}.

\begin{figure}[H] 
\centering
\includegraphics[scale=1]{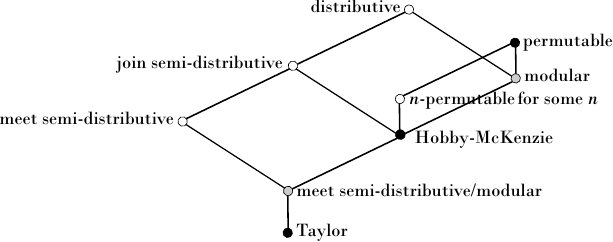}
\caption{Maltsev filters that establish a hierarchy of varieties.} 
\label{fig:hierarch}
\end{figure}


The structure of this article  is as follows. In Section 2, we give the basic definitions related to algebras, varieties, clones, and relational structures. The lattice of interpretability types of varieties and  the filter of Hobby-McKenzie varieties are defined. Some basic facts on Hobby-McKenzie varieties are presented.  In Section 3, we state and prove some general lemmas on clone homomorphisms. We need these lemmas for our characterizations  of Hobby-McKenzie varieties given in Section 5. We opted here for a much more general approach than in~\cite{BGMZ} in the hope that our general theorems obtained in Section 3  will be used to prove primeness for other Maltsev filters in the future. In Section 4, we introduce certain ternary structures called  partial semilattices and prove some statements related to them. Partial semilattices play a major role in  our characterizations in Section 5. In Section 5, we characterize the Hobby-McKenzie varieties, and prove our main result that the filter of Hobby-McKenzie varieties is prime in $\LL$.

\section{Preliminaries}

We introduce some basic concepts, definitions, and some well-known facts that we shall use in the later sections.

\subsection{Algebras and varieties}

Let $A$ be a set and $n$ a non-negative integer. A map from $A^n$ to $A$ is called an {\em $n$-ary operation} on $A$. A subset of $A^n$ is called an {\em $n$-ary relation}. We say that an $n$-ary operation $f$ {\em preserves} a $k$-ary relation $R$ if for any 
\begin{gather*}r_1= (r_{1,1},\dots, r_{1,k}),\dots, r_n=(r_{n,1},\dots, r_{n,k})\in R,\\  
f(r_1,\dots, r_n)\coloneqq (f(r_{1,1},\dots, r_{n,1}),\dots,f(r_{1,k},\dots, r_{n,k}))\in R.
\end{gather*}  
 
A {\em signature (for algebras)} is a set where a non-negative integer is assigned to each element of the set. The elements of a signature are called {\em function symbols}. If $f$ is a function symbol and $n$ is assigned to $f$, then $n$ is called the {\em arity} of $f$. We also say that $f$ is an {\em $n$-ary function symbol} in this case.
Let $A$ be a set, and  for each function symbol $f$ in a given signature, let $f'$ be an operation of the same arity as $f$ on $A$. 
Then the set $A$ with the operations $f'$ is called an {\em algebra in the given signature}. The set $A$ is called the {\em  base set} (or {\em underlying set}) of the algebra, and the operations $f'$ are called the {\em basic operations} of the algebra. Throughout the paper, we use bold face capitals and the same capitals to denote algebras and their base sets, respectively. Usually, we denote the basic operation $f'$ of an algebra $\alg A$  by  $f^{\alg A}$ for any function symbol $f$ in the signature.

Let $\alg A$ and $\alg B$ be two algebras of the same signature. We say that $\alg B$ is a {\em subalgebra} of $\alg A$ if $B\subseteq A$, and the basic operations of $\alg A$ preserve $B$, and the basic operations of $\alg B$ are those of $\alg A$, restricted to $B$. Let $C\subseteq A$. Then the smallest subalgebra that contains $C$ in  $\alg A$ is called the {\em subalgebra generated} by $C$. We say that $C$ {\em generates} $\alg A$ if $\alg A$ equals the subalgebra generated by $C$. In this case,  $C$ is called a {\em generating set of}  $\alg A$. An algebra is {\em finitely generated} if it has a finite generating set.

A map $\nu\colon A\to B$ is a {\em homomorphism} from $\alg A$ to $\alg B$ if for any function symbol $f$, say $n$-ary, and any $a_1,\dots,a_n\in A$, $$\nu(f^{\alg A}(a_1,\dots,a_n))=f^{\alg B}(\nu(a_1),\dots,\nu(a_n)).$$ A bijective homomorphism is called an {\em isomorphism}. An algebra $\alg B$ is a {\em homomorphic image} of $\alg A$ if there exists an onto homomorphism from $\alg A$ to $\alg B$. An algebra $\alg A$ is {\em isomorphic} to $\alg B$ if there exists an isomorphism from $\alg A$ to $\alg B$.

A congruence $\varrho$ of an algebra $\alg A$ is an equivalence on its base set such that the basic operations of $\alg A$ preserve $\varrho$. If $\alg A$ is an algebra of some signature and $\varrho$ is a
congruence of $\alg A$, then we define the {\em quotient algebra} $\alg A/\varrho$ of the same signature to be the algebra whose base set is the set of the blocks of $\varrho$  and whose basic operations are defined for any function symbol $f$, say $n$-ary, and any $a_1,\dots ,a_n\in A$ by
$$ f^{\alg A/\varrho}(a_1/\varrho,\dots ,a_n/\varrho)\coloneqq f^{\alg A}(a_1,\dots ,a_n)/\varrho$$
where for any $a\in A$, $a/\varrho$ denotes the block of $a$. Now, it is immediate from the definition that $\alg A/\varrho$ is a homomorphic image of  $\alg A$ under the homomorphism $a\mapsto~ a/\varrho$. It is well known, cf. Theorem 1.16 in~\cite{MMT}, that the kernel of any homomorphism from $\alg A$ to $\alg B$ is a congruence of $\alg A$ and that the homorphic images of any algebra $\alg A$ coincide with the quotient algebras of $\alg A$ up to isomorphism.

Let $\alg A_i$, $i\in I$, be algebras of the same signature. Their {\em product} $\prod\limits_{i\in I}\alg A_i$ is the algebra whose base set equals $\prod\limits_{i\in I}A_i$ and whose basic operations are defined as follows: for any function symbol $f$, say $n$-ary,  
$$f^{\prod\limits_{i\in I}\alg A_i}((a_{1,i})_{i\in I},\dots, (a_{n,i})_{i\in I})\coloneqq (f^{\alg A_i}(a_{1,i},\dots, a_{n,i}))_{i\in I}.$$

Let $S$ be any set, sometimes  called the {\em set of variables}. In a given signature, the $S$-ary terms are the finite sequences of symbols given by the following recursive definition:
\begin{enumerate}
\item the elements of $S$ are all $S$-ary terms,
\item for any function symbol $f$, say $n$-ary, and $S$-ary terms $t_1,\dots, t_n$, the term $f(t_1,\dots, t_n)$ is also an $S$-ary term.
\end{enumerate}
  
Let $\alg A$ be an algebra and $t$ be an $S$-ary term in the same signature. Then $t$  determines a map $t^{\alg A}\colon A^S\to A$ that is called the {\em $S$-ary  term operation} related to $t$  in $\alg A$  and  defined recursively as follows: 
\begin{enumerate}
\item if $t=s$ for some $s\in S$, then $t^{\alg A}$ equals the  projection to the $s$-th coordinate,
\item if $t=f(t_1,\dots, t_n)$ for an $n$-ary function symbol $f$ and $S$-ary terms $t_1,\dots, t_n$, then
$$t^{\alg A}\coloneqq f^{\alg A}(t_1^{\alg A},\dots, t_n^{\alg A}).$$                                         
\end{enumerate}

A pair of two $S$-ary terms put in the form $t=u$ is called an {\em $S$-ary identity}. An algebra $\alg A$ {\em satisfies} $t=u$ if $t^{\alg A}=u^{\alg A}$. Let $\Sigma$ be a set of $S$-ary identities. An algebra $\alg A$ {\em satisfies} (or {\em is a model of}) $\Sigma$ if $\alg A$ satisfies all the identities in $\Sigma$.
 
\begin{definition}\label{var}                 
A  class of algebras of the same signature is called a {\em variety} (or an {\em equational class}) if there is a set of $S$-ary identities whose models coincide with the members of the class. In this case, we also say that the set of the given identities {\em defines} the variety. Since the  satisfaction of an identity is determined by a finite subset of $S$, $S$ can always be assumed to be countably infinite.  We say that a variety $\var V$ satisfies a set $\Sigma$ of identities if any algebra in $\var V$ satisfies $\Sigma$.
A variety is  called {\em  locally finite} if its finitely generated members are finite. A variety is called {\em  trivial} if it only contains one-element algebras. A variety $\var W$ is a {\em subvariety} of a variety $\var V$ if they have the same signature and $\var W$ satisfies all identities satisfied by $\var V$.\vege
\end{definition}

Birkhoff's famous theorem in~\cite{Bi} asserts that the varieties are exactly the classes of algebras of the same signature that are closed under product, subalgebra, and homomorphic image of algebras.   The variety defined by the set of identities $$\{x\wedge y=y\wedge x,\ (x\wedge y)\wedge z= x\wedge (y\wedge z),\ x\wedge x=x\}$$  in a signature with a single binary operation symbol $\wedge$ plays a major role in the present article. This variety is called the {\em variety of semilattices} and denoted by $\mathcal{S}\mathcal{L}$.

\begin{definition}\label{freealg}
Let $S$ be a set, and let $T_S$ denote the set of $S$-ary terms in a given signature. The {\em term algebra} $\alg T_S$  over $S$  is the algebra with the base set $T_S$ and basic operations defined for every function symbol $f$, say $n$-ary, in the signature and every $t_1,\dots, t_n\in T_S$ by
$$f^{\alg T_S}(t_1,\dots, t_n)\coloneqq f(t_1,\dots, t_n).$$ It is clear from the definition that for any $S'\subseteq S$, $\alg T_{S'}$ is a subalgebra of $\alg T_S$. Let $\var V$ be a variety in a given signature. Then the {\em free algebra} $\alg F_{\var V}(S)$ in $\var V$ equals the factor algebra $\alg T_S/\Theta_S$ where $\Theta_S$ is the set of pairs $(t,u)$ for which $\var V$ satisfies the $S$-ary identity $t=u$. It can be shown that $\Theta_S$ is indeed a congruence of $\alg T_S$, and hence, $\alg T_S/\Theta_S$ is well defined.
Clearly, $\alg F_{\var V}(S)\in \var V$ also holds.  The blocks $s/\Theta_S$ where $s\in S$ generate $\alg F_{\var V}(S)$.
If $\var V$ is non-trivial, then these blocks are pairwise different.  This set of blocks is called a {\em free generating set of $\alg F_{\var V}(S)$}, since for any two $S$-ary terms $t$ and $u$, the identity $t=u$ is satisfied by $\var V$ if and only if the values of the term operations $t^{\alg F_{\var V}(S)}$ and $u^{\alg F_{\var V}(S)}$ coincide at the tuple $(s/\Theta_S)_{s\in S}$. 
An algebra isomorphic to $\alg F_{\var V}(S)$ is called the {\em free algebra with free generating set $S$ in $\var V$}.\vege
\end{definition}

It is useful to conceive of the free algebra with free generating set $S$ in $\var V$ in the following two ways.
\begin{enumerate}
\item We pick a term from each block of $\Theta_S$, and we think of the set of these designated $S$-ary terms as the base set of the free algebra in $\var V$. A basic operation $f$ on this set works like in $\alg T_S$, except that any value of $f$ in $\alg T_S$ is replaced by the designated $S$-ary term from the $\Theta_S$-block of the value.
\item We may think of the base set of the free algebra in $\var V$ as the set of the $S$-ary term operations of $\alg F_{\var V}(S)$. Then the basic operations of $\alg F_{\var V}(S)$ act naturally  on the set of the $S$-ary term operations, which yields the basic operations of  the free algebra.             
\end{enumerate}
For further topics on varieties, the interested reader may consult~\cite{B},~\cite{BS}, and~\cite{MMT}.

\subsection{Clones and interpretability}

A {\em clone on} $A$ is a set  of finitary operations on $A$ that contains all projection operations and is closed under composition. A {\em clone homomorphism} from a clone $\mathscr C$ to a clone $\mathscr D$ is a map from $\mathscr C$ to $\mathscr D$ that maps each projection of $\mathscr C$ to the corresponding projection of $\mathscr D$ and commutes with composition. Let $\alg A$ be an algebra. The {\em clone of} $\alg A$ denoted by $\clo (\alg A)$ is the clone of finitary term operations of $\alg A$. The {\em clone of a variety} $\var V$ is the clone of the free algebra with countably infinite free generating set.

\begin{definition}\label{interpr}
A {\em variety $\var V$ interprets in a variety $\var W$} if there is a clone homomorphism from the clone  of $\var V$ to the clone of $\var W$. We write $\var V\leq \var W$ in notation. A {\em  term reduct of a variety} $\var W$ is a variety  whose signature is a set $T$ of some terms of $\var W$, and whose identities are those that are satisfied by $\var W$ for the terms in $T$.\vege
\end{definition}

There is a way to express interpretation of varieties in terms of subvariety and term reduct: $\var V\leq\var W$ if and only if there is a variety $\var V'$ of the same signature as $\var V$ such that $\var V'$ is a subvariety of $\var V$ and $\var V'$ is a term reduct of $\var W$.
To verify the ``only if'' part of the latter statement, let us suppose that $\var V\leq \var W$, and that this is witnessed by the clone homomorphism $\alpha$.
Let $\omega$ denote the set of non-negative integers. Then the operations $\alpha(f^{\alg F_{\var V}(\omega)})$ where $f$ runs through the signature of $\var V$ yield an algebra $\alg A$ on the base set of $\alg F_{\var W}(\omega)$. Since $\alpha$ is a homomorphism, $\alg A$ satisfies all identities of $\var V$. On the other hand, the $\alpha(f^{\alg F_{\var V}(\omega)})$ are term operations of the free algebra $\alg F_{\var W}(\omega)$, hence, the set $\Sigma$ of identities given by the $\alpha(f^{\alg F_{\var V}(\omega)})$ is satisfied by each algebra in $\var W$. Since $\Sigma$ contains the identities satisfied by $\var V$, it defines a subvariety $\var V'$ of $\var V$. Clearly, $\var V'$ also is a term reduct of $\var W$. The ``if''  part of the statement also holds, since any variety interprets in each of its subvarieties, and any term reduct of a variety $\var W$ interprets in $\var W$.

For an example of an interpretation, we consider two varieties. Let $\mathcal{S}\mathcal{E}\mathcal{T}$ be the {\em variety of sets}, the variety defined by the empty set of identities in the empty signature. Let $\mathcal{S}\mathcal{G}$ be the {\em variety of semigroups}, the variety defined by the single identity $(xy)z=x(yz)$  in the signature of one binary operation symbol. Clearly, there is a clone homomorphism from the clone of $\mathcal{S}\mathcal{E}\mathcal{T}$ to the clone of $\mathcal{S}\mathcal{G}$, mapping the projections to the corresponding projections. So, $\mathcal{S}\mathcal{E}\mathcal{T}\leq \mathcal{S}\mathcal{G}$. On the other hand, there is also a clone homomorphism from the clone of $\mathcal{S}\mathcal{G}$ to the clone of $\mathcal{S}\mathcal{E}\mathcal{T}$, mapping the product operation to the projection to the first coordinate. Thus,  $ \mathcal{S}\mathcal{G}\leq\mathcal{S}\mathcal{E}\mathcal{T}$ also holds. 
Hence, by the definition just follows,  $\mathcal{S}\mathcal{E}\mathcal{T}$ and $ \mathcal{S}\mathcal{G}$ have the same interpretability type.

\begin{definition}\label{lattice}
As easily seen, interpretability is a quasiorder on the class of varieties. The blocks of this quasiorder are called  {\em interpretability types}. By following~\cite{GT}, we define the {\em lattice $\LL$ of interpretability types of varieties} that is obtained by taking the quotient of the class of varieties quasiordered by interpretability and the related equivalence.

The join in  $\LL$ is described as follows. Let $\var{V}_1$  and $\var{V}_2$ be two varieties of disjoint signatures. Let $\var{V}_1$  and $\var{V}_2$  be defined by the sets $\Sigma_1$ and $\Sigma_2$ of identities, respectively. Their {\em  join} $\var{V}_1\vee\var{V}_2$ is the variety defined by the set of identities $\Sigma_1\cup\Sigma_2$ in the signature that equals the union of the signatures $\var{V}_1$  and $\var{V}_2$. The so defined join is compatible with the interpretability relation of varieties, and naturally yields the join operation in $\LL$.

In this article, the meet operation of $\LL$ is not used, but for completeness we describe it. The meet $\var V\wedge\var W$ of  varieties $\var {V}$  and $\var W$ has the signature  $\cup_{n=0}^\infty T_n\times S_n$ where $T_n$ and  $S_n$ are the sets of $n$-ary terms of $\var V$ and  $\var W$, respectively.  For any $\alg A\in \var V$  and $\alg B\in \var W$, we define an algebra $\alg C$  with base set $A\times B$ where for any 
$(s,t)\in  T_n\times S_n$, $a_1,\dots,a_n\in A$, and $b_1,\dots,b_n\in B$
$$(s,t)^{\alg C}((a_1,b_1),\dots,(a_n,b_n))= (s^{\alg A}(a_1,\dots,a_n), t^{\alg B}(b_1,\dots,b_n)). $$
Now, $\var V\wedge\var W$  is the variety that consists of the isomorphic copies of the so defined algebras $\alg C$. The meet of varieties is compatible with the interpretability relation of varieties, and naturally yields the meet operation in $\LL$.\vege
\end{definition}

We note that the interpretability types of varieties actually form a proper class. Also, $\LL$ has a smallest and a largest element. The smallest element of $\LL$ is the interpretability type of the variety of sets. The largest element of $\LL$ is the interpretability type that contains the trivial varieties.

\begin{definition}\label{Maltsev} 
A {\em strong Maltsev condition} is a finite set of identities in some signature. A {\em Maltsev condition} is a sequence $M_n,\ 0\leq n,$ of strong  Maltsev conditions where the variety defined by $M_{n+1}$ interprets in the variety defined by $M_n$ for each $n\geq 0$. We say that a {\em variety $\var V$ admits} this Maltsev condition   if the variety defined by some of the  $M_n$ interprets in $\var V$.  In this case, we also say that $\var V$  {\em admits the terms} of the given Maltsev condition. An upwardly closed sublattice in a lattice is meant by a {\em filter}. Clearly, the types of the varieties that admit a particular Maltsev condition form a filter in $\LL$.  Such filters of $\LL$ are called {\em Maltsev filters}. A Maltsev filter is called {\em a strong Maltsev filter} if it consists of the types of the varieties that admit a particular strong Maltsev condition\vege
\end{definition}

\begin{example}\label{msl}
The following sequence of sets of identities determines a Maltsev condition. For any non-negative integer $n$, let $M_n$ be the set of identities:
\begin{itemize}
\item[] $f_0(x,y,u,v)=x,$
\item[] $f_i(x,y,x,y)=f_{i+1}(x,y,x,y)$ and
\item[] $f_i(x,x,y,y)=f_{i+1}(x,x,y,y)$ for even $0\leq i\leq 2n$,
\item[] $f_i(x,x,x,y)=f_{i+1}(x,x,x,y)$ for odd $0\leq i\leq 2n$,
\item[] $f_{2n+1}(x,y,u,v)=v.$
\end{itemize}

To see that the above sequence $M_n,\ n\geq 0$ is indeed a Maltsev condition, just observe that the variety $\var V$ defined by $M_n$ admits $M_{n+1}$. This fact is witnessed by the 4-ary terms $f_0(x,y,u,v),\dots,f_{2n+1}(x,y,u,v) $ and $f_{2n+2}=f_{2n+3}=v$ of $\var V$.\vege
\end{example}

A filter is {\em proper} if it is non-empty and it differs from the whole lattice. A proper filter $F$ in a lattice is called {\em prime} if for any two elements $a,b\not\in F$,  we have $a\vee b\not\in F$. To gain information on the structural properties of $\LL$, it is useful to determine which of the well-known classical Maltsev filters are prime filters in $\LL$. In case a Maltsev filter is (not) prime, we just say sometimes that the related Maltsev condition is (not) prime.

\subsection{Hobby-McKenzie varieties}

Let $A$ be a set. An $n$-ary operation $f$ on $A$ is {\em idempotent} if $f(a,\dots,a)=a$ for all $a\in A$. An {\em algebra  is idempotent} if its all basic operations are idempotent.  An $n$-ary  term $t$ of a variety $\var V$ is called an {\em idempotent term} if $\var V$ satisfies the identity $t(x,\dots,x)=x$. A {\em variety $\var V$ is idempotent} if all algebras in the variety are idempotent, or equivalently, for all $n$, all $n$-ary terms of $\var V$ are idempotent.
 
\begin{definition}\label{HMv}
An identity is {\em linear} (or {\em of height at most one}) if it has at most one occurrence of a function symbol on each side of the identity. A {\em Hobby-McKenzie term} of a variety $\var V$ is an $n$-ary idempotent term $t$ for some $n$ such that  for each non-empty subset $I$ of $\{1,\dots, n\}$, $\var V$ satisfies a linear identity of the form $t(\dots)=t(\dots)$ in two different variables $x$ and $y$ where for each $i\in I$, the variable in the $i$-th position on the left side equals $x$, and for some $i\in I$, the variable in the $i$-th position on the right side equals $y$. In other words, $t$ is idempotent and satisfies a set of linear identities that fails to be satisfied by any term of $\mathcal{S}\mathcal{L}$. A variety is a {\em  Hobby-McKenzie variety} if it admits a Hobby-McKenzie term. We name the interpretability types of these varieties {\em Hobby-McKenzie interpretability types}. In~\cite{KK}, the Hobby-McKenzie varieties were characterized by various equivalent conditions. In particular, it was shown that a variety $\var V$ is Hobby-McKenzie if and only if $\var V$ admits the Maltsev condition in Example~\ref{msl}. So, the Hobby-McKenzie interpretability types form a Maltsev filter in $\LL$. This filter  is called the  {\em Hobby-McKenzie filter}.\vege 
\end{definition}

In a lattice, a filter is called a {\em principal filter} if it consists of all elements that are greater than or equal to some element of the lattice. Clearly, any strong  Maltsev filter is a principal filter in $\LL$. On the other hand, any Maltsev filter that is a principal filter is easily seen to be a strong Maltsev filter. In~\cite{KKVW}, it was shown that there is no strong Maltsev condition that characterizes the class of Hobby-McKenzie varieties. In other words, the Hobby-McKenzie filter is not a strong Maltsev filter or, equivalently, it is not a principal filter in $\LL$.

The {\em  full idempotent reduct of a variety} $\var V$ is the variety  whose signature is the set of idempotent terms of $\var V$, and whose identities are those satisfied by $\var V$ for its idempotent terms. Let $\var V_{\Id}$ denote the full idempotent reduct of $\var V$. We require the following characterization of Hobby-McKenzie varieties.

\begin{theorem}[cf. Lemma 9.5  in~\cite{HM}]\label{THM} Let $\var V$ be a variety. The following are equivalent.
\begin{enumerate}
\item $\var V$  is a Hobby-McKenzie variety.
\item $\var V_{\Id}\not\leq \mathcal{S}\mathcal{L}.$
\end{enumerate}
\end{theorem}

By the use of the preceding theorem, it follows that if $\var V_1$ and $\var V_2$ are idempotent non-Hobby-McKenzie varieties, then they both interpret in $\mathcal{S}\mathcal{L}$. So, their join also interprets in  $\mathcal{S}\mathcal{L}$. Hence, in the sublattice of interpretability types of idempotent varieties, the filter of the types of idempotent Hobby-McKenzie varieties is prime. The main result of this paper states that we can drop idempotency from this statement, that is, the Hobby-McKenzie filter is prime in $\LL$.

\subsection{Relational structures and connectivity}

A {\em signature (for relational structures)} is a set where a non-negative integer is assigned to each element of the set. The elements of a signature are called {\em relation symbols}. If $R$ is a relation symbol and $n$ is assigned to $R$, then $n$ is called the {\em arity} of $R$. We also say that $R$ is an {\em $n$-ary relation symbol} in this case.
Let $G$ be a set, and  for each relation symbol $R$ in a given signature, let $R'$ be a relation of the same arity as $R$ on $G$. 
Then the set $G$ with the relations $R'$ is called a {\em relational structure in the given signature}. The set $G$ is called the {\em  base set} of the relational structure. Throughout the paper, we use blackboard bold capitals and the same capitals to denote relational structures and their base sets, respectively. For any relation symbol $R$ in the signature of a relational structure $\mG$, we denote the corresponding relation $R'$ of $\mG$ by $R^{\mG}$. On some occasions, though, we leave off the superscript to ease up notation.

A map $\varphi\colon G\to H$ is called a {\em homomorphism from $\mG$ to $\mH$} if 
$\varphi$ maps any tuple of any relation of $\mG$ to a tuple of the corresponding relation of $\mH$. If $\varphi$ is a homomorphism from $\mG$ to $\mH$, we also say that $\varphi$ {\em preserves} the relations of $\mG$.
For a homomorphism $\varphi\colon \mG\to \mH$, we denote by $\varphi(\mG)$ the relational structure whose base set is $\varphi(G)$ and whose relations are the images of the relations  of $\mG$ under $\varphi$. We say that $\mH$ is a {\em weak substructure} of $\mG$ if every element of $\mH$ is an element of $\mG$ and every tuple  of a relation of $\mH$ is in the corresponding relation of $\mG$.  We use $\mH\subseteq\mG$ to denote that $\mH$ is a weak substructure of $\mG$. For example,  $\varphi(\mG)\subseteq \mH$ for any homomorphism $\varphi\colon \mG\to \mH$. We call $\mH$ a {\em substructure of }$\mG$, if $\mH\subseteq\mG$ and every relation of $\mH$ coincides with the restriction of the corresponding relation of $\mG$ to $H$.
 
A homomorphism  $\alpha\colon \mG \to \mH$ is a {\em retraction}  if there is a homomorphism $\beta\colon \mH \to~\mG$ such that  $\alpha\beta$ coincides with the identity map on $H$. Then  the homomorphism $\beta$ is called a {\em coretraction} and the structure $\mH$ is called a {\em retract} of $\mG$.  A bijective retraction is called an {\em isomorphism}. The structures $\mG$ and $\mH$ are {\em isomorphic} if there exists an isomorphism from $\mG$ to $\mH$. We remark that if $\mH$ is a retract of $\mG$  with a coretraction $\beta$, then $\beta(\mH)$ is a substructure isomorphic to $\mH$ in $\mG$.

Let $\varrho$ be an equivalence of a relational structure $\mG$. We define the {\em quotient structure} $\mG/\varrho$ of the same signature as $\mG$ to be the relational structure whose base set is the set of the blocks of $\varrho$,  and for any relational symbol $R$,  $R^{\mG/\varrho  }$ consists of the tuples $r/\varrho$ where $r\in R^{\mG}$ and each component of $r/\varrho$ is the block of the related component of $r$.  We note that the map $a \mapsto a/\varrho$ is a homomorphism from $\mG$ to $\mG/\varrho$. Moreover, for any homomorphism $\varphi\colon \mG\to \mH$, $\varphi(\mG)$ is isomorphic to $\mG/\ker(\varphi)$ where $\ker(\varphi)$ denotes the kernel of $\varphi$.
 
Let $I$ be an arbitrary set, and let $\mH_i$ be relational structures of the same signature where $i\in I$. The {\em disjoint union} $\dot\bigcup_{i\in I}\mH_i$ of the $\mH_i$ is the relational structure of the same signature as the $\mH_i$, $i\in I$, whose base set is the disjoint union of the base sets $H_i$, $i\in I$, and whose relations are the disjoint unions of the corresponding relations of the $\mH_i$, $i\in I$. 
The {\em product} $\prod\limits_{i\in I}\mH_i$ of the $\mH_i$ is the relational structure of the same signature as the $\mH_i$, $i\in I$, whose base set equals $\prod\limits_{i\in I} H_i$ and whose relations are defined as follows: for any relation symbol $R$, say $k$-ary,  $$((h_{1,i})_{i\in I},\dots, (h_{k,i})_{i\in I})\in R^{\prod\limits_{i\in I} \mH_i}\text{ iff } (h_{1,i},\dots, h_{k,i})\in R^{\mH_i}  \text{ for all }{i\in I}.$$ The {\em $I$-th power} $\mH^I$ of $\mH$ is defined to be the relational structure $\prod\limits_{i\in I}\mH_i$ where $\mH_i=\mH$ for all $i\in I$.   A homomorphism from $\mH^n$ to $\mH$ is called an $n$-ary {\em polymorphism} of $\mH$. 
We note that the polymorphisms of $\mH$ coincide with the finitary operations of $H$ that preserve the relations of $\mH$. Let $\pol(\mH)$ denote the clone of polymorphisms of $\mH$. 

\begin{definition}\label{comp}
We say that a {\em relational structure $\mG$  is compatible in an algebra} $\alg A$ if  $G=A$ and the basic operations of $\alg A$ preserve the relations of $\mG$. A {\em relational structure is compatible in a variety} if it is compatible in an algebra of the variety.~\vege\end{definition}

It is easy to see that  a relational structure $\mG$ is compatible in a variety $\var V$ if and only if there is a clone homomorphism from $\clo(\var V)$ to $\pol(\mG)$. Hence, if $\var V\leq \var W$ and $\mG$ is compatible in $\var W$, then $\mG$ is also compatible in $\var V$.

We also have that if $\mG$ is compatible in the varieties $\var V_1$ and $\var V_2$ of disjoint signatures, then  $\mG$ is compatible in $\var V_1\vee \var V_2$ as well.  Indeed, the premise gives that for $1\leq i\leq 2$, there exist $\alg A_i\in \var V_i$ such that $A_i=G$ and the basic operations of $\alg A_i$ preserve the relations of $\mG$. Now, let $\alg B$ denote the algebra on $G$ with basic operations  $f_i^{\alg A_i}$  for any $1\leq i\leq 2$ where $f_i$ is an arbitrary function symbol in the signature of $\var V_i$. Then $\alg B\in \var V_1\vee \var V_2$ and $\mG$ is compatible in $\alg B$, and hence, in $\var V_1\vee \var V_2$ as well.

Now we introduce a notion of connectivity  for relational structures, and then, we shall  prove a related lemma that will be frequently used in the later sections.
	
\begin{definition} We define the {\em binary projection} $\mB$ of a relational structure $\mG$ as follows: $\mB$ is  a binary relational structure on $G$ such that the signature of $\mB$ consists of the binary relational symbols $R_{i,j}$ where $R$ is a $k$-ary relational symbol in the signature of $\mG$ and $1\leq i<j \leq k$, and the relations of $\mB$ are the
$$R_{i,j}^\mB=\{(a_i, a_j)\colon (a_1, \dots, a_k) \in R^{\mG}\}.$$  We note that  $\pol(\mG)\subseteq \pol(\mB)$. An {\em oriented path} of $\mB$ is a sequence of elements $b_0\dots,b_t$ in $B$ where for all $0\leq s\leq t-1$ either $(b_s,b_{s+1})$ or $(b_{s+1},b_s)$ is in some relation of $\mB$. We define the {\em connectivity equivalence} $\alpha$  of a relational structure $\mG$ by $$(a,b)\in \alpha\text{ if and only if there is an oriented path  from } a\text{ to }b\text{ in }\mB.$$  It is  easy to see that $\alpha$ indeed is an equivalence on $G$. The blocks of the connectivity equivalence are called the ({\em connected}) {\em components} of $\mG$. Clearly, every relational structure $\mG$ may be regarded as the disjoint union of its components. We say that $\mG$ {\em  is connected} if it has precisely one component.  Clearly, if $\mG$ is connected, then for any homomorphism $\alpha\colon \mG\to\mH$, $\alpha(\mG)$ is also connected.\vege
\end{definition}

A relation on $G$ is called {\em reflexive} if it contains all of the constant tuples of $G$. A {\em relational structure  is reflexive} if all of its relations are reflexive. Clearly, if $\mG$ is reflexive, then its binary projection $\mB$ is reflexive, and for any homomorphism $\alpha\colon \mG\to\mH$, $\alpha(\mG)$ is also reflexive. We require the following lemma on connectivity of reflexive structures.

\begin{lemma}\label{connected} The following hold.
\begin{enumerate}
\item Let $\mH_i$ be reflexive relational structures of the same signature for $1\leq i\leq n$. Let $\mH_i=\dot\bigcup_{u\in U_i}\mH_{i,u}$, where the $\mH_{i,u}$, $u\in U_i$, are the components of $\mH_i$. Then the components of $\prod\limits_{i=1}^{n} \mH_i$ are the relational structures of the form $\prod\limits_{i=1}^{n} \mH_{i,u_i}$ where $u_1\in U_1,\dots u_n\in U_n$ are arbitrary elements.

\item Let $\mG$ be a reflexive relational structure. Then the connectivity equivalence of $\mG$ is preserved by the polymorphisms of 
$\mG$. 
 
\item Let $I$ be a set and $\mH$ a reflexive connected relational structure. If $I$ or $\mH$ is finite, then $\mH^I$ is a reflexive connected relational structure.

\item Let $K$ be a finite set of reflexive connected relational structures of the same signature such that all powers of them are connected. Then any product with (possibly multiple) factors from $K$ is connected.

\item Let $K$ be a finite set of reflexive connected relational structures of the same signature such that all powers of them are connected.
Let $\mH=\prod\limits_{i\in I}\mH_i$ be a relational structure where for any $i\in I$, the components of $\mH_i$ are products of relational structures from $K$. Then the components of $\mH$ are of the form $\prod\limits_{i\in I} \mC_i$ where $\mC_i$ is a component of $\mH_i$ for any $i\in I$.

\end{enumerate}
\end{lemma}
\begin{proof} 
The proof of the lemma is based on the following claim.
\begin{cla}
Let $\mB_i$ be arbitrary reflexive connected binary structures of the same signature where $1\leq i\leq k$. Let $a_i$ and $b_i$ be arbitrary elements in $\mB_i$ for any $1\leq i\leq k$. Then there exist $m$, relation symbols $R_1,\dots, R_m$, and for each $1\leq i\leq k$, an $(a_i,b_i)$-path in $\mB_i$ of the form $a_i=c_{i,0},\dots,c_{i,m}=b_i$ such that  for all $1\leq j\leq m$ either for all $1\leq i\leq k$, $(c_{i,j-1},c_{i,j})\in R_j$ or for all $1\leq i\leq k$, $(c_{i,j-1},c_{i,j})\in R^{-1}_j$.
\end{cla}
\begin{claimproof} 
To prove the Claim, we may start with some $(a_i,b_i)$-paths $P_i$ of length $\ell_i$ in $\mB_i$ where $1\leq i\leq k$.  We think of $P_i$ as a labeled path in $\mB_i$ where any edge is labeled by a  relation symbol $R$ or $R^{-1}$, meaning that the edge is contained in $R^{\mB_i}$ or $(R^{\mB_i})^{-1}$.  We shall synchronize the $P_i$, that is, we construct new $(a_i,b_i)$-paths of the same length such that all have the same type of edge labeling. First, we synchronize the paths $P_1$ and $ P_2$. We make a new $(a_1, b_1)$-path $P'_1$ from $P_1$ by adding $\ell_2$ loop edges to $P_1$ at the end of $P_1$. The labeling of the new edges by relation symbols  mirrors back the labeling of the edges of $P_2$.  Since $\mB_1$ is reflexive,  $P'_1$ is indeed a new labeled $(a_1, b_1)$-path in $\mB_1$. Now, we make a new labeled $(a_2, b_2)$-path $P'_2$ from $P_2$ by adding $\ell_1$ loop edges to $P_2$ at the beginning of $P_2$. The labeling of the new edges  mirrors back the labeling of the edges of $P_1$. Since $\mB_2$ is reflexive,  $P'_2$ is indeed a new labeled $(a_2, b_2)$-path in $\mB_2$. Now, it is clear that both $P'_1$ and $P'_2$ have length $\ell_1+\ell_2$ and their corresponding edges are labeled by the same relation symbol $R$ or $R^{-1}$. By iterating this procedure, in $k-1$ steps, we get to new labeled $(a_i,b_i)$-paths that are all synchronized, which finishes the proof.
\end{claimproof}

To prove (1), all we have to show is that the $\prod\limits_{i=1}^{n} \mH_{i,u_i}$ are connected. (Clearly, each tuple of each relation of $\prod_{i=1}^n\mH_i$ lies entirely in one such product, so the connected components cannot be larger.)

Let $\mG_i=\mH_{i,u_i}$ where $1\leq i\leq n$. We know that these relational structures are connected. Let $\mB_i$ be the binary projection of $\mG_i$ where $1\leq i\leq n$ and $\mB$ the binary projection of $\prod\limits_{i=1}^{n} \mG_i$. Since the $\mG_i$ are connected, the $\mB_i$ are also connected. 
Clearly, $\mB=\prod\limits_{i=1}^{n} \mB_i$. We apply the Claim to the  $\mB_i$. So, for each $1\leq i\leq k$ there exists an $(a_i,b_i)$-path in $\mB_i$ such that these paths have the same length and their related edges have the same labeling by relations of the $\mB_i$. By the definition of product, this means that there is  a same type of labeled path connecting $(a_1,\dots, a_n)$ and $(b_1,\dots, b_n)$ in $\prod\limits_{i=1}^{n} \mB_i=\mB$. So, $\prod\limits_{i=1}^{n} \mG_i$ is connected, which concludes the proof of (1).

To prove (2), let $f$ be an $n$-ary polymorphism  of $\mG$, and  for any $1\leq i\leq n$ let $\mG_i$ be an arbitrary component of $\mG$. By item (1), the product $\prod\limits_{i=1}^n\mG_i$ is connected. Since $f$ is a polymorphism, $f(\prod\limits_{i=1}^n\mG_i)$ is also connected. Thus, $f$ preserves connectivity.

The claim when $I$ is finite in item (3), is an immediate consequence of item (1). For the case, when $\mH$ is finite, let $\mB$ be the binary projection of $\mH$. Then the binary projection of $\mH^I$ equals $\mB^I$. For any two elements $(a_i)_{i\in I}$ and $(b_i)_{i\in I}$ of $B^I$ and for any $i\in I$, we take an $(a_i,b_i)$-path in $\mB$ such that the number of these paths altogether is finite. This can be done since $\mB$ is connected and finite. By the use of the Claim, we may assume that these finitely many paths are synchronized. Now, these synchronized $(a_i,b_i)$-paths witness the fact that that there is a path from $(a_i)_{i\in I}$ to $(b_i)_{i\in I}$ in $\mB^I$. Thus, $\mH^I$ is connected.

Now, we prove item (4). Clearly, any product with factors from $K$ is isomorphic to the product of finitely many powers of relational structures from $K$. Since these powers are reflexive and assumed to be connected, their product is also connected by item (1).

For item (5), it suffices to prove that $\mC=\prod\limits_{i\in I} \mC_i$ is connected. As the $\mC_i$ are isomorphic to products with factors from $K$, $\mC$ is also isomorphic to a product of same kind.  By item (4), such a product is connected, hence, $\mC$ is connected.
\end{proof}

We note that  if $K$ is a finite set of finite reflexive connected relational structures, then by item (3) of the lemma, $K$ satisfies the conditions of items (4) and (5).

\section{Free relational structures and clone homomorphisms}

We require some lemmas in connection with clone homomorphisms. For our characterizations of Hobby-McKenzie varieties in Section 5, we need to construct a particular relational structure on the underlying set of an appropriate algebra of a non-Hobby-McKenzie variety. Our construction is a general one and may  apply not only for non-Hobby-McKenzie varieties but for other varieties as well. In the following definitions and two lemmas, we delineate this general construction and describe some of its properties that we need later.

Recall that after Definition~\ref{freealg}, we gave two ways how to conceive of a free algebra $\alg F$. We shall use both approaches at the same time. So, we conceive of an element of $F$ as an $S$-ary term $f$, as well as, the related  $S$-ary term operation $f^{\alg F}$. Usually, the superscript is left off.

\begin{definition}\label{defofF} 
Let $\var V$ be a non-trivial variety, and let $\mS$ be any relational structure. Let $\alg F\in\var V$ denote the free algebra with free generating set $S$, and $\mF$ the relational structure on $F$ whose relations coincide with the subalgebras generated by the relations of $\mS$ in suitable powers of $\alg F$. We call $\mF$ the {\em free relational structure generated by $\mS$ in $\var V$}. Since the relations of $\mF$ are subalgebras of powers of $\alg F$, the basic operations of $\alg F$ preserve the relations of $\mF$, that is, $\mF$ is a compatible relational structure in $\alg F$, and hence in $\var V$. Let $U$ denote the set of unary term operations of $\alg F$, and for any $u\in U$ let $\mF_u$ be the substructure induced in $\mF$ by the elements $f$ for which $f(\bar x)=u(x)$ for all $x\in F$ where $\bar x$ denotes the constant $S$-tuple whose entries are all equal to $x$. Before proving our first lemma in this section, we give an example of a free relational structure generated by $\mS$ in $\var V$. 
\vege\end{definition}

\begin{example}\label{freerel} We assume now that $\var V$ is the variety of commutative semigroups satisfying the extra identity $x^2y=xy$, and $\mS$ is the reflexive 3-cycle, that is, $\mS$ is the digraph on $S=\{x,y,z\}$ with the single relation $$\{(x,x),(y,y),(z,z),(x,y),(y,z),(z,x)\}.$$ Notice that the two-element semilattice and two-element set with a single constant binary operation are members of  $\var V$. Let $\alg F$ be the free algebra with free generating set $\{x,y,z\}$ in $\var V$. Then by using the identities defining $\var V$ and the particular algebras just given in $\var V$, one can easily see that $\alg F$ has the 10-element base set $$F=\{x,y,z,x^2,y^2,z^2,xy,xz,yz,xyz\}.$$

The edges of the free relational structure $\mF$ generated by $\mS$ in $\var V$ are the elements of the binary relation generated by $$\{(x,x),(y,y),(z,z),(x,y),(y,z),(z,x)\}$$ in $\alg F^2$. Thus, they are of the form $(t(x,y,z,x,y,z),t(x,y,z,y,z,x))$ where $t$ is a 6-ary term operation of $\alg F$. These term operations are related to the 6-ary terms that are described similarly to the elements of $F$. Thus, up to equivalence, the  6-ary terms of $\alg F$ are  $x_1^2,\dots,x_6^2$ and the products with pairwise different  factors from $\{x_1,\dots,x_6\}$. By taking this into account, the edges of $\mF$ are easily obtained.

We also note that $\alg F$ has two unary term operations related to the unary terms $x$ and $x^2$. The first of these unary operations is the identity map $\id$ on $F$, and the second one is denoted by $r$ where $r$ maps $x,y,z$ to $x^2,y^2,z^2$ respectively, and fixes the other elements of $F$. So, $U=\{\id, r\}$. Hence, $F_{\id}=\{x,y,z\}$ and  
\[F_{r}=\{x^2,y^2,z^2,xy,xz,yz,xyz\}. \]\vglue-16truept\vege
\end{example}

The free relational structure $\mF$ in the example is displayed in Figure~\ref{fig:free}. Note that for clarity, we left off loops, and used undirected edges to replace back and forth edges in the figure. We note that the clone of $\var V_{\Id}$ consists of only projections, and hence, there is a clone homomorphism from it to $\pol(\mS)$. Observe that the components of $\mF$ are $\mF_{\id}$ and  $\mF_r$, and  $\mS$ is a retract of $\mF_{\id}$. The following lemma shows that a similar phenomenon holds in general.

\begin{figure}[H] 
\centering
\includegraphics[scale=1]{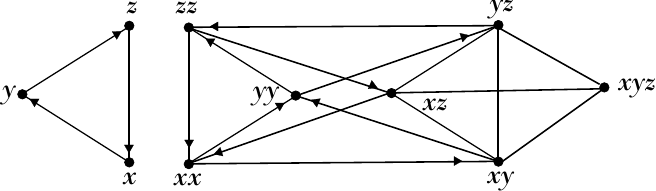}
\caption{A free relational structure $\mF$.} 
\label{fig:free}
\end{figure}


\begin{lemma}\label{retr} 
Let $\mS$ be a reflexive connected relational structure. Let $\var V$ be a non-trivial variety, $\alg F\in\var V$ the free algebra with free generating set $S$, and $\mF$  the free relational structure generated by $\mS$ in $\var V$. Let $U$ denote the set of unary term operations of $\alg F$, and for any $u\in U$, let $\mF_u$ be the substructure defined in Definition~\ref{defofF}. Then the following hold.
 
\begin{enumerate}
\item  The components of $\mF$ coincide with the substructures $\mF_u$, $u\in U$. For the identity operation $\id$ of $F$,  $\mF_{\id}$ is the component induced by the  idempotent $S$-ary term operations in  $\mF$.
\item If there is a clone homomorphism from the clone of $\var V_{\Id}$ to $\pol(\mS)$, then  $\mS$ is a retract of $\mF_{\id}$ with the coretraction $\iota \colon S\to  F_{\id},\ a\mapsto a.$
\end{enumerate}
\end{lemma} 
\begin{proof} 
First we prove (1).  Let $f\in F_u$. As $f$ is an $S$-ary term operation, there is a finite set $P\subseteq S$ and a $P$-ary term operation $f'$ such that $f(\mS^S)=f'(\mS^P)$. Then by item (3) of Lemma~\ref{connected}, $\mS^{P}$ is connected. Since $f'$ preserves connectedness, $f'(\mS^{P})=f(\mS^S)$ is also connected. Since $f\in F_u$, for any $x\in S$
$$u(x)=f(\bar x)\in f(S^S).$$ Clearly, we also have that $f\in f(S^{S}).$ So, by taking some $x\in S$, $u(x)$ is connected to any $f\in F_u$ via a  path in $\mF$. Hence,  $\mF_u$ is connected.

We prove that if $f$ and $g$ are in the same component of $\mF$, then $f(\bar x)=g(\bar x)$ for all $x\in F$. It suffices to prove that for any two entries $f$ and $g$ of any tuple in any relation of $\mF$, $f(\bar x)=g(\bar x)$. By the definition of the relations of $\mF$, there is a (say $n$-ary) term operation $h$ of $\alg F$ such that $h(a)=f$ and $h(b)=g$ for some $n$-tuples $a,b\in S^n$. These equalities yield two identities that hold in $\var V$. By identifying the variables in these identities, we obtain that $f(\bar x)=h(\hat x)=g(\bar x)$ where $\hat x$ denotes the constant $n$-tuple whose entries are all equal to $x\in F$.

To finish the proof of item (1), observe that $\mF_{\id}$ is the component induced by the  $S$-ary  term operations $f^{\alg F}$ of $\alg F$ such that $f$ is an $S$-ary term and $f^{\alg F}(\bar x)=x$ holds in $\alg F$. Thus, $\var V$ satisfies the identity $f(\bar x)=x$.
This means that  $f^{\alg F}\in \mF_{\id}$ if and only if $f$ is an idempotent term. 

Now we prove item (2). Since $S$ generates $\alg F$, each element of $\mF_{\id}$ is of the form $t^{\alg F}(s_1,\dots,s_n)$ where $n$ is a non-negative integer, $t$ is an $n$-ary term of $\var V$, and  $s_1,\dots,s_n$ are some elements in $S$.  Observe that $t$ is idempotent by the second claim of item (1). We also may assume that $s_1,\dots,s_n$ are pairwise different. 

Let $\alg E=\alg F_{\var V_{\Id}}(\omega)$, and let $\mathscr C$ denote the clone of $\var V_{\Id}$, that is, $\mathscr C=\clo(\alg E)$. Let $^*$ be a clone homomorphism from $\mathscr C$ to $\pol(\mS)$. 
We claim that $$\mu\colon F_{\id}\to S,\ t^{\alg F}(s_1,\dots,s_n)\mapsto (t^{\alg E} )^*(s_1,\dots,s_n),$$ where $n$ is a non-negative integer, $t$ is an $n$-ary idempotent term  of $\var V$ and $s_1,\dots,s_n$ are pairwise different elements in $S$,  is a well-defined homomorphism. Suppose that $v^{\alg F}(s_1,\dots,s_n)= w^{\alg F}(s_1,\dots,s_n)$ where $v$ and $w$ are two $n$-ary idempotent terms of $\var V$ and $s_1,\dots,s_n$ are pairwise different elements in $S$. Then the identity $v= w$ holds in $\var V$. Since $v$ and $w$ are idempotent terms, $v= w$ holds in $\var V_{\Id}$ as well. Hence, $v^{\alg E}~=~w^{\alg E}.$ Therefore, $(v^{\alg E})^*=(w^{\alg E})^*$, and so, $(v^{\alg E})^*(s_1,\dots,s_n)= (w^{\alg E})^*(s_1,\dots,s_n)$ holds in $\mS$. Thus, $\mu$ is a well-defined map.

Let $R$ be any $k$-ary relation of $\mS$, let $R'$ be the corresponding relation of $\mF$, and let $R''=R'|_{F_{\id}}$.
We show that $\mu$ preserves the relation $R''$. We take a typical $k$-tuple in $R''$ as $v^{\alg F}(r_1,\dots,r_m)$ where $r_1,\dots,r_m\in R$ and $v$ is an $m$-ary term of $\var V$. 

While the term $v$ is not {\em a priori} idempotent, we still argue that it is. Let $v^{\alg F}(s_{1},\dots,s_{m})$ be the first entry of the $k$-tuple $v^{\alg F}(r_1,\dots,r_m)$ where $s_{1},\dots,s_{m}\in S$ are the first entries of the tuples $r_1,\dots,r_m$ respectively. Since $v^{\alg F}(s_{1},\dots,s_{m})\in F_{\id}$,  it equals $t^{\alg F}(s'_1,\dots,s'_n)$ for some $n$-ary idempotent term $t$ of $\var V$ and $s'_{1},\dots,s'_{n}\in S$. Then $\var V$ satisfies the identity $v(x_{1},\dots,x_{m})=t(y_1,\dots,y_n)$ where $x_i=y_j$ if and only if $s_i=s'_j$. Therefore, $\var V$ satisfies the identity $v(x,\dots,x)=t(x,\dots,x)$. By the idempotency of $t$,
$\var V$ also satisfies the identity $t(x,\dots,x)=x$. Thus, $\var V$ satisfies $v(x,\dots,x)=x$, that is, $v$ is idempotent.

Then, the image of the $k$-tuple $v^{\alg F}(r_1,\dots,r_m)$ under $\mu$ is $(v^{\alg E})^*(r_1,\dots,r_m)$. As $(t^{\alg E})^*\in\pol (\mS)$ and $r_1,\dots,r_m\in R$, $(t^{\alg E})^*(r_1,\dots,r_m)\in R$. So, $\mu$ preserves $R''$, and hence, $\mu$ is a homomorphism. Since $\iota$ is a homomorphism and $\mu\iota$ is the identity map on $S$, we have that $\mu$ is a retraction,  $\iota$ is a coretraction and $\mS$ is a retract of $\mF_{\id}$.
\end{proof}

The preceding lemma, for a non-trivial variety $\var V$ and a reflexive relational structure $\mS$, yields some useful properties of a special compatible relational structure in $\var V$. Our next lemma gives similar properties for, in some sense,  an even nicer compatible  relational structure in $\var V$. 

\begin{definition}\label{defofK}
Let  $\mF$ and $\mS$ be relational structures of the same signature, and let  $\mF_u$, $u\in U$, be the components of $\mF$. For any $u\in U$, let $H_u$ denote the set of non-constant homomorphisms from $\mF_u$ to $\mS$. For any $u\in U$ and $t\in F_u$, let $[t]$ be the map $H_u\rightarrow S, \varphi\mapsto \varphi(t)$. If $H_u=\emptyset$, then $[t]$ is just the empty map. Let $$\psi\colon F\to \dot\bigcup_{u\in U}S^{H_u}, \ t\mapsto [t].$$
We remark that $\psi$ is a homomorphism from $\mF$ to $\dot\bigcup_{u\in U}\mS^{H_u}$ where $\mS^{H_u}$ is the one-element reflexive structure if $H_u=\emptyset$. Indeed, $\psi$ is the union of the maps $\psi|_{F_u}\colon F_u\to S^{H_u},\ t\mapsto 
[t]$,   $ {u\in U}$, and these maps are homomorphisms from $\mF_u$ to $\mS^{H_u}$ as their coordinate maps $\varphi:\mF_u\to \mS,\ t\mapsto \varphi(t)$, $\varphi\in H_u$, are ones. Let $$\mK\coloneqq\psi(\mF)\subseteq \dot\bigcup_{u\in U}\mS^{H_u},$$ and for any $u\in U$, $\mK_u\coloneqq\psi(\mF_u)\subseteq\mS^{H_u}$. The following example is to demonstrate the new notions defined in this paragraph for the relational structure $\mF$ constructed in Example~\ref{freerel}.
\vege\end{definition}

\begin{example} In Example~\ref{freerel}, we established that the set of unary term operations of $\alg F$ is $U=\{\id,r\}$. Since $\mF_{\id}=\mS$, $H_{\id}=\{\id_{F_{\id}},\varphi, \varphi^2\}$ where $\varphi$ is the permutation of $\{x,y,z\}$ that maps $x$ to $y$ and $y$ to $z$. On the other hand, every homomorphism from $\mF_{r}$ to $\mS$ is constant, so $H_r=\emptyset$. Let $c$ denote the only element of $S^{H_r}$. Then the above defined map $\psi$ is as follows
$$\psi\colon F\to S^{\{\id_{F_{\id}},\varphi, \varphi^2\}} \cup \{c\}, \ t\mapsto [t]$$ where $[t]=(t, \varphi(t),\varphi^2(t))$ if $t\in F_{\id}$, and $[t]= c$ if $t\in F_r$. Then $$K=\{(x,y,z), (y,z,x), (z,x,y), c\},\ K_{\id}=\{(x,y,z), (y,z,x), (z,x,y)\},\ K_{r}=\{c\},$$ and the edges of $\mK$ are all loop edges in $K$ and 
\[(x,y,z)\to (y,z,x)\to (z,x,y)\to (x,y,z).\]
The resulting relational structure $\mK$ is depicted in Figure~\ref{fig:K} where loop edges are not displayed.
\vege
\end{example}

\begin{figure}[H] 
\centering
\includegraphics[scale=1]{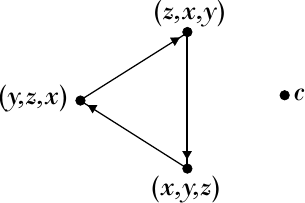}
\caption{The relational structure $\mK$.} 
\label{fig:K}
\end{figure}

In the following lemma, we establish some properties of the  map $\psi$ and relational structure $\mK$ defined in Definition~\ref{defofK} if they come from a free relational structure $\mF$ generated by $\mS$ in a non-trivial variety $\var V$.

\begin{lemma}\label{retract2}
Let $\var V$ be a non-trivial variety, and $\mS$ a reflexive connected relational structure. Let $\alg F\in\var V$ denote the free algebra with free generating set $S$, and $\mF$  the free relational structure generated by $\mS$ in $\var V$.  Let $U$ be the set of unary term operations of $\alg F$, and $\mF_u$, $u\in U$, the components of $\mF$. Let $\mK$, $\psi$, and for any $u\in U$, $H_u$ and $\mK_u$  be as defined in Definition~\ref{defofK}.  Then the following hold.
\begin{enumerate}
    
    \item $\mK$ and $\mK_u$ where $u\in U$ are reflexive relational structures.
    \item The components of $\mK$ coincide with the substructures $\mK_u$ of $\mK$ where $u\in U$.
     \item If there is a clone homomorphism from the clone of $\var V_{\Id}$ to $\pol(\mS)$, then
		$\mS$ is a retract of $\mK_{\id}$ with the coretraction $\iota'\colon S\to \{[a]\colon a\in S\},\ a\mapsto [a]$.
    \item For any $u\in U$, every homomorphism $q\colon \mK_u\to \mS$ is a constant or a projection to a unique coordinate $\varphi\in H_u$.
    \item The kernel $\ker(\psi)$ of $\psi$ is a congruence of $\alg F$.
    \item Let $\alg K$ be the algebra in $\var V$ whose underlying set is $K$ and whose basic operations are defined by 
    $$t^{\alg K}([x_1],\dots,[x_n])\coloneqq [t^{\alg F}(x_1,\dots,x_n)]$$ 
    for all basic operations $t^{\alg F}$ of $\alg F$ and $x_1,\dots,x_n\in F$ provided $t^{\alg F}$ is $n$-ary. Then 
    the operations of  $\alg K$ are well-defined, $\alg K$ is a homomorphic image of $\alg F$ under $\psi$, and $\mK$ is a compatible relational structure in $\alg K$.
\end{enumerate}
\end{lemma}
\begin{proof}
  Item (1) is straightforward, for $\mK=\psi(\mF)$, $\mF$ is reflexive, and $\psi$ preserves reflexivity for its image. Item (2) is also not hard. As $\psi$ is a homomorphism, it preserves connectivity. So, the substructures $\mK_u=\psi(\mF_u)$, where $u\in U$, are connected. Since $\psi(\mF_u)\subseteq \mS^{H_u}$ for each $u\in U$ and the copies $\mS^{H_u}$, $u\in U$, are disjoint in $\dot\bigcup_{u\in U}\mS^{H_u}$, the $\mK_u$ coincide with the components of $\mK$.

To prove item  (3), we define the map $\mu'\colon K_{\id}\to S,\ [t]\mapsto \mu(t)$ where $\mu\colon \mF_{\id}\to \mS$ is the retraction defined in the proof of item (2) of the preceding lemma. Notice that $\ker( \psi|_{F_{\id}}) \subseteq\ker( \mu)$, hence, $\mu'$ is well-defined. Now, it is clear that $\mu'$ is a homomorphism. Since $\iota'=\psi\iota$, where  $\iota$ is the coretraction in the statement of item (2) of the preceding lemma, is a homomorphism and $\mu'\iota'$ is the identity map on $S$, we have that $\mu'$ is a retraction, $\iota'$ is a coretraction and $\mS$ is a retract of $\mK_{\id}$.

To prove item (4), let $\varphi=q \psi|_{F_{u}}$. Suppose that $q$ is not constant. By using that $\psi(F_{u})=K_{u}$, $\varphi$ is not constant, and hence, $\varphi\in H_{u}$. Then for any element $t\in F_{u}$ $$q([t])=q(\psi(t))=\varphi(t)=[t](\varphi).$$ The uniqueness of $\varphi$ now is obvious by the equalities $\varphi(t)= q([t])$ for all $t\in F_u$.

Now, we prove item (5).  It is enough to show that $\ker(\psi)$ is preserved by every unary polynomial operation of $\alg F$, cf. e.g., Theorem 4.18 in~\cite{MMT}. Every unary polynomial operation of $\alg F$ is of the form $t(x,t_2,\dots,t_n)$ where $t_2,\dots,t_n\in F$ and $t$ is an $n$-ary term operation of $\alg F$. Observe that by the reflexivity of $\mF$, every unary polynomial operation of $\alg F$ is an endomorphism of $\mF$. Now, let us assume that $(v,w)\in \ker(\psi)$. We prove that $(t(v,t_2,\dots,t_n),t(w,t_2,\dots,t_n))$ is also in $\ker(\psi)$.

By the definition of $\psi$, there is a $u\in U$ such that $v,w\in F_u$. Since $\mF$ is reflexive and $t\in\pol(\mF)$, by item (2) of Lemma~\ref{connected}, $t$ preserves the connectivity of $\mF$. Hence, there exists $u'\in U$ such that $t(v,t_2,\dots,t_n),t(w,t_2,\dots,t_n)\in F_{u'}$.

In the case of $H_{u'}=\emptyset$, this yields immediately that $$(t(v,t_2,\dots,t_n),t(w,t_2,\dots,t_n))\in  \ker(\psi).$$ Now, suppose that $H_{u'}$ is non-empty, and let $\varphi\in H_{u'}$ be arbitrary. Then the map $$\nu\colon F_u\mapsto \mS, x\mapsto \varphi(t(x,t_2,\dots,t_n))$$ is a homomorphism from $\mF_u$ to $\mS$. Since $(v,w)\in \ker(\psi)$, this implies in particular that $\nu(v)=\nu(w)$. Therefore, $\varphi(t(v,t_2,\dots,t_n))=\varphi(t(w,t_2,\dots,t_n))$. Since this holds for all $\varphi\in H_{u'}$, we obtain that  $(t(v,t_2,\dots,t_n),t(w,t_2,\dots,t_n))\in  \ker(\psi)$.

Now, it is not hard to see that item (6) holds. Indeed, item (5) immediately yields that $\alg K$ is well-defined, and hence, $\alg K$ is a homomorphic image of $\alg F$ under $\psi$. The latter and $\mK=\psi(\mF)$ easily imply that $\mK$ is a compatible structure of the algebra  $\alg K=\psi(\alg F)$ as well.  
\end{proof}

In Section 5, we also use the following lemma on clone homomorphisms. 
\begin{lemma}[see Lemma 4.1 in~\cite{BGMZ}]\label{clonehom} Let $\mG$ be a relational structure and $\mK\subseteq\mG$. Let $\mathscr C$ be a subclone of $\pol(\mG)$. If for any $f\in \pol(\mK)$, there is unique extension $f^*\in \mathscr C$, then the map $$\varphi\colon \pol(\mK)\to \mathscr  C,\ f\mapsto f^*$$ is a clone homomorphism.
\end{lemma}

\section{Partial semilattices}

In the rest of this paper, we mainly consider relational structures which are ternary and have a single relation. In Section 5, we shall see that in our characterizations of Hobby-McKenzie varieties, certain ternary structures called  partial semilattices play a crucial role. A ternary structure $\mG$ is a {\em partial semilattice} if it has a single  reflexive ternary relation that consists of some  triples of the form $(a,b,a\wedge b )$, $a,b\in G$ where  $\wedge$ is a binary meet semilattice operation  with an underlying set containing $G$.

Let $\mG$ be a partial semilattice, and let $(a,b,c)$ be a triple of $\mG$. Thro\thanks{The author has been funded by the European Research Council (Project POCOCOP, ERC Synergy Grant 101071674). Views and opinions expressed are however those of the authors only and do not necessarily reflect those of the European Union or the European Research Council Executive Agency. Neither the European Union nor the granting authority can be held responsible for them.}ughout this paper, we denote the unique element $c$ by $a\sqcap b$. Let $a_i\in G$, $1\leq i\leq n$. We write $\bigsqcap\limits_{i=1}^n a_i$ for
$$((\dots(a_1\sqcap a_2)\dots )\sqcap a_{n-1})\sqcap a_n$$
if all of 
$$a_1\sqcap a_2, (a_1\sqcap a_2)\sqcap a_3,\dots, ((\dots(a_1\sqcap a_2)\dots)\sqcap a_{n-1})\sqcap a_n$$ 
exist in $\mG$.
It is not required by definition that a partial semilattice contain a triple $(a,b, c)$ for any  $a,b\in G$, but if it does, we call it a {\em full semilattice}. Notice that if $f$ is a homomorphism from $\mG$ to $\mH$ and $a\sqcap b$ exists in $\mG$, then $f(a)\sqcap f(b)$ also exists in $\mH$ and  $f(a\sqcap b)=f(a)\sqcap f(b)$.

Let $\mG$ be a partial semilattice. We call an element $1\in G$ a {\em largest element of $\mG$}  if for any $a\in G$, $(a,1, a)$ and $(1,a, a)$ are triples in $\mG$.  Clearly,  in any partial semilattice, the largest element is unique if there exists one.
It is also easy to see that the product $\mG\times\mH$ of two partial semilattices $\mG$ and $\mH$ is a partial semilattice, and if $1^{\mG}$ is the largest element of $\mG$ and  $1^{\mH}$ is the largest element of $\mH$, then $(1^{\mG},1^{\mH})$ is the largest element of $\mG\times\mH$. We require the following lemma on partial semilattices.

\begin{lemma}\label{prod_of_partial_sl}
Let $\mH_i$, $1\leq i\leq n$, be partial semilattices with largest elements $1_1,\dots, 1_n$, respectively. Let $\mH$ be a partial semilattice, and $f$ a homomorphism from  $\prod\limits_{i=1}^n\mH_i$ to $\mH$. Let $f_i$ be the unary map obtained from $f(x_1,\dots,x_n)$ by replacing each $x_j$ by $1_j$ for all $j\neq i$. Then for any $1\leq i\leq n$, $f_i$ is a homomorphism from $\mH_i$ to $\mH$. Moreover, $f(x_1,\dots,x_n)=\bigsqcap\limits_{i=1}^n f_i(x_i)$.
 \end{lemma} 
\begin{proof}
Notice  that by reflexivity, the constant self-maps of $\mH_i$, $1\leq i\leq n$, are homomorphisms.  The unary maps $f_i$ are homomorphisms as they are obtained from the homomorphism $f$ and the constant homomorphisms $1_j$, $j\neq i$, by composition.
We prove the second claim of the lemma by induction on $n$. When $n=1$, the claim is obvious. Let $n>1.$
By our earlier remark, $\mM=\prod\limits_{i=1}^{n-1}\mH_i$ is a partial semilattice with largest element $1=(1_1,\dots,1_{n-1})$. Moreover, $f(x_1,\dots,x_{n-1},1_n)$ is a homomorphism from  $\mM$ to $\mH$. So, $f(x_1,\dots,x_{n-1},1_n)=\bigsqcap\limits_{i=1}^{n-1} f_i(x_i)$ by the induction hypothesis. From now on, we naturally identify  $\prod\limits_{i=1}^n\mH_i$  with the relational structure $\mM\times \mH_n$. Let $x=(x_1,\dots,x_{n-1})\in M$ and $x_n\in H_n$. Then, by using that $1$ is the largest element of $\mM$ and $1_n$ is the largest element of $\mH_n$,  $(x,1,x)$ is a tuple of $\mM$ and $(1_n,x_n,x_n)$ is tuple in $\mH_n$. So, $((x,1_n), (1,x_n),(x,x_n))$  is a tuple of $\mM\times \mH_n$.
Thus, by using that $f$ is a homomorphism, 
$(f(x,1_n), f(1,x_n),f(x,x_n))$ is a tuple of $\mM\times \mH_n$, that is,
$f(x,x_n)=f(x,1_n)\sqcap f(1,x_n)$. So, by taking into account the equalities $f(x,1_n)=\bigsqcap\limits_{i=1}^{n-1} f_i(x_i)$ and $f(1,x_n)=f_n(x_n)$, we obtain that $f(x_1,\dots,x_n)=\bigsqcap\limits_{i=1}^n f_i(x_i)$.
\end{proof}

Let $\mS$ denote the 2-element full semilattice structure on $\{0,1\}$ with the single relation 
$\{(0,0,0),\ (0,1,0),\ (1,0,0),\ (1,1,1)\}.$ Notice that the relation of  $\mS$  is just the \lq\lq graph\rq\rq of the meet semilattice operation on $\{0,1\}$, and $1$ is the largest element of $\mS$.

We note that if the partial semilattice $\mH$ in the statement of the preceding lemma is a weak substructure of a full semilattice structure whose related semilattice operation is denoted by $\wedge$ (in a typical application of the lemma in this paper, $\mH$ will be a weak subpower of $\mS$), then we may replace $\sqcap$ by $\wedge$ in the statement of the lemma. Now, by the help of the preceding lemma, $\pol(\mS)$ for the above 2-element structure $\mS$ is easily described.

\begin{corollary} \label{cloneofs} Let $\mS$ be the relational structure on $\{0,1\}$ with the single relation 
$\{(0,0,0),\ (0,1,0),\ (1,0,0),\ (1,1,1)\}.$ Then the  clone $\pol(\mS)$ consists of all constant operations of $\{0,1\}$ and all the meet semilattice  operations of all arities on $\{0,1\}$. 
\end{corollary}

\section{Characterizations of Hobby-McKenzie varieties}

In this section, we prove the main result of the paper. The following characterization of Hobby-McKenzie varieties plays a crucial role in the proof of our main result. Throughout this section, $\mS$ denotes the relational structure on $\{0,1\}$ with the single relation $\{(0,0,0),\ (0,1,0),\ (1,0,0),\ (1,1,1)\}.$

\begin{lemma}\label{majdnem_jo}
Let $\var V$ be a variety. Then $\var V$ is a non-Hobby-McKenzie variety if and only if there exist a non-empty  set $U$ and sets $H_u,\ u\in U$, such that some of the $H_u$ are non-empty and  $\dot\bigcup_{u\in U}\mS^{H_u}$ is a compatible structure in $\var V$. If  $\var V$ is a locally finite non-Hobby-McKenzie variety, then $U$ and for all $u\in U,$  $H_u$ can be chosen to be finite.
\end{lemma}

\begin{proof} 
First we are going to prove the ``if'' part of the first claim of the lemma. We assume that $\dot\bigcup_{u\in U}\mS^{H_u}$ is a compatible relational structure in $\var V$ and for some $u\in U$, $H_u\neq\emptyset$. Then $\mS$ is a retract of $\dot\bigcup_{u\in U}\mS^{H_u}$. It is easy to see that linear identities are preserved for retract. So, the existence of a Hobby-McKenzie polymorphism is inherited under taking retract.  Thus, by using that $\mS$ admits no Hobby-McKenzie operation, $\var V$ is a non-Hobby-McKenzie variety.

Now we prove the ``only if'' part of the first claim. Our goal is to construct a compatible relational structure  of the required form in a non-Hobby-McKenzie variety $\var V$. Before proceeding, we delineate the  main steps of our proof.
First, we use Lemmas~\ref{retr} and~\ref{retract2} to construct
a specific compatible structure $\mK$ in $\var V$. It turns out that the components of $\mK$ naturally embed into powers of $\mS$. So $\mK$ itself naturally embeds into a structure $\mG$ that is a disjoint union of $\mS$-powers. Thus, $\mG$ has the form required by the lemma.
It remains  to verify that $\mG$ is a compatible relational structure in $\var V$.  In order to do this, we invoke Lemma~\ref{clonehom} to prove that there is a clone homomorphism from $\pol(\mK)$ to a subclone of $\mG$. To meet the requirements of the lemma, we define a subclone $\mathscr C$ of $\pol(\mG)$ and we prove that every polymorphism of $\mK$ extends uniquely to an operation of $\mathscr C$ (see Claim 4 below).  In order to verify unique extendability, we  need some information on the shape of the components of $\mK$ (see Claims 1 and 2) and the polymorphisms of $\mK$ (see Claim 3). Finally, we argue at the end of the proof that if there is a clone homomorphism from $\pol(\mK)$ to $\mathscr C$, then $\mG$ is a compatible structure in $\var V$.

So, let us assume that $\var V$ is a non-Hobby-McKenzie variety. By Theorem~\ref{THM}, this is equivalent to the fact that the full idempotent reduct $\var V_{\Id}$ of $\var V$ interprets in $\mathcal{S}\mathcal{L}$.  Let  $\alg F$ be the free algebra freely generated by two elements $x$ and $y$ in $\var V$. As $\alg F$ is the free algebra with a two-element free generating set, we conceive of the underlying set $F$ of $\alg F$ as a set of designated binary terms, as well as, the set of binary term operations of $\alg F$, see the comment after Definition~\ref{freealg}.

We define a compatible reflexive ternary structure $\mF$ in $\alg F$ by letting the only  relation of $\mF$ be the subalgebra of $\alg F^3$ generated by the set of triples $$\{(x,x,x),\ (x,y,x), \ (y,x,x), \ (y,y,y)\}.$$  By definition, for any $t,v,w\in F$, $(t,v,w)\in\mF$ if and only if there is a 4-ary term $h$ of $\var V$ such that
$$t=h(x,x,y,y),\ v=h(x,y,x,y)\text{ and  }w=h(x,x,x,y)\text{ in } \alg F.$$ By Definition~\ref{defofF}, $\mF$ basically is  the free relational structure generated by  $\mS$ in $\var V$. The only difference is that the free generators of $\alg F$ are now called $x$ and $y$ instead of $0$ and $1$, respectively. To make it precise, $\mF$  is  the free relational structure generated by 
$$\mS_\mF\coloneqq(\{x,y\};\ \{(x,x,x),\ (x,y,x), \ (y,x,x), \ (y,y,y)\})\text{ in }\var V.$$  Clearly,  $\mS_\mF$ is isomorphic to $\mS$.

Let $U$ be the set of unary term operations of $\alg F$. For any $u\in U$, we write $F_u$ for the set of the binary term operations $t\in F$ such that $t(x,x)=u(x)$. For any $u\in U$, let $\mF_u$ be the substructure of $\mF$ induced by $F_u$.

Now we check that the premises of items (1) and (2) of Lemma~\ref{retr} for $\var V$ and $\mS_\mF$ are satisfied. The variety $\var V$ is a non-trivial variety, since $\var V$ is non-Hobby-McKenzie. Also, there is a clone homomorphism from the clone of $\var V_{\Id}$ to $\pol(\mS_\mF)$.  Indeed, $\var V_{\Id}$ interprets in $\mathcal{S}\mathcal{L}$, $\mS_\mF\cong \mS$, and by Corollary~\ref{cloneofs}, $\pol(\mS)$  consists of  the meet semilattice operations on the set $\{0,1\}$.

Therefore, by item (1) of Lemma~\ref{retr}, the $\mF_u$ are exactly the connected components of $\mF$, and  for the identity map $\id\in U$, the base set of $\mF_{\id}$ consists of  the idempotent term operations in $F$. Moreover, by item (2), $\mS_\mF$ and, hence, $\mS$ is a retract of $\mF_{\id}$.

For any $u\in U$, let $H_u$ denote the set of non-constant homomorphisms from $\mF_u$ to $\mS$.  For any $t\in F_u$, we write $[t]$ for the map $H_u\rightarrow S, \varphi\mapsto \varphi(t)$. One may consider $[t]$ as an $H_u$-ary tuple $(\varphi(t))_{\varphi\in H_u}.$
Now, by the remark in Definition~\ref{defofK}, the map $$\psi\colon \mF\to \dot\bigcup_{u\in U}\mS^{H_u}, \ t\mapsto [t] $$ is a homomorphism. Let $\mK\coloneqq\psi(\mF)$. In other words, $\mK$ is the ternary structure with base set $\{[t]\colon\ t\in F\}$  whose relation consists of the triples $([t],[v],[w])$ where $(t,v,w)$ is a triple in $\mF$. Let $\mK_u\coloneqq \psi(\mF_u)\subseteq \mS^{H_u}$ where $u\in U$. 
 
Now, we apply Lemma~\ref{retract2}. By item (1), $\mK $ and the $\mK_u$, $u\in U$, are all reflexive structures.
By item (2), the components of $\mK$ coincide with the relational structures
$\mK_u$ where $u\in U$. Clearly, $\mK=\dot\bigcup_{u\in U}\mK_u\subseteq \dot\bigcup_{u\in U}\mS^{H_u}$. 
By item (3), $\mS$ is a retract of $\mK_{\id}$  with the coretraction $0\mapsto [x]$ and $1\mapsto [y]$. So, the substructure induced by  the set $\{[x],[y]\}$ in $\mK_{\id}$ is isomorphic to $\mS$. We write $\mS^{\mK}$ for this substructure.

We define an algebra $\alg K$ in $\var V$ whose underlying set is $K$ and whose basic operations are defined by $$t^{\alg K}([x_1],\dots,[x_n])\coloneqq [t^{\alg F}(x_1,\dots,x_n)]$$ for all basic operations $t^{\alg F}$ of $\alg F$ and $x_1,\dots,x_n\in F$ provided $t^{\alg F}$ is $n$-ary. By items (5) and (6) of Lemma~\ref{retract2}, we know that $\ker(\psi)$ is a congruence of $\alg F$, $\alg K$ is well-defined and  is a homomorphic image of $\alg F$ under $\psi$. We also have that $\mK$ is a compatible structure of $\alg K$. So, $\mK$ is compatible in $\var V$. We establish  some additional properties of $\mK$ through the next two claims.

 \begin{cla1} For any $u\in U$, $\mK_u$ is a partial semilattice with the largest element $[u(y)]$. \end{cla1}
 \begin{claimproof} By $\mK_u\subseteq \mS^{H_u}$, the structures $\mK_u$ are partial semilattices for all $u\in U$. We prove that $[u(y)]$ is the largest element of $\mK_u$. We know that $(y,x,x)$ and $(y,y,y)$ are  triples of $\mF$. As any $t\in F_u$ preserves the relation of $\mF$, $(t(y,y), t(x,y), t(x,y))$ is a triple of $\mF$. We similarly get that $(t(x,y), t(y,y), t(x,y))$ is also a triple of $\mF$. The map $\psi$ also preserves the relation, hence, both $$([t(y,y)], [t(x,y)], [t(x,y)])\text{ and }([t(x,y)], [t(y,y)], [t(x,y)])$$ are triples in $\mK_u$. So, $[u(y)]=[t(y,y)]$ is the largest element of $\mK_u$. \end{claimproof}

\begin{cla2} For any $u\in U$, $|K_u|=1$ if and only if $[u(x)]=[u(y)]$. If $|K_u|\geq 2$, then the substructure induced by $\{[u(x)], [u(y)] \}$ in $\mK_u$ is isomorphic to $\mS$. \end{cla2}
\begin{claimproof} Suppose that $[u(x)]=[u(y)]$, and let $t\in K_u$ be arbitrary.

By using the facts that $$t_1=([t(y,y)], [t(x,y)], [t(x,y)])\text{ and } t_2=([t(x,x)], [t(x,y)], [t(x,x)])$$ are triples of $\mK_u$, $$[t(x,x)]=[u(x)]=[u(y)]=[t(y,y)],$$ and that $\mK_u$ is a partial semilattice,  the first two entries of $t_1$ and $t_2$ are equal and so, the third ones are also equal, that is, $$[t(x,y)]=[t(x,x)]=[u(x)].$$ Thus, $|K_u|=1$.

If $|K_u|\geq 2$, then by $u$ being an endomorphism of $\mK$, the  2-element substructure induced by $\{[u(x)], [u(y)] \}$ in $\mK_u$ contains   $u(\mS^{\mK})$. By using that $\mK_u$ is a partial semilattice,  $u(\mS^{\mK})$ coincides with the substructure induced by $\{[u(x)], [u(y)] \}$  in $\mK_u$, and hence, the latter is isomorphic to $\mS$. \end{claimproof}

We use  the preceding two claims and  Lemma~\ref{prod_of_partial_sl} to get detailed information on the polymorphisms of $\mK.$
From this point on, we write $U^+$ for the set of those $u\in U$ for which $|K_u|\geq 2$ or equivalently $H_u\neq \emptyset$. Now, let $f$ be an $n$-ary operation in $\pol(\mK)$. Clearly, $f$ restricted to any connected component of $\mK^n$ is a homomorphism to some component of $\mK$. 

Since  the $\mK_u$ are the components of $\mK$ and they are reflexive, by item (1) of Lemma~\ref{connected}, it is also clear that any component $\mD$ of $\mK^n$ equals some product $\prod\limits_{\ell=1}^n \mK_{u_{\ell}}$ where $u_1,\dots,u_n\in U$. So, $f|_D$ is a homomorphism from $\prod\limits_{\ell=1}^n \mK_{u_{\ell}}$ to some $\mK_{u_0}, u_0\in U$. If $u_0\in U^+$, since $\mK_{u_0}\subseteq \mS^{H_{u_0}}$, $f|_D=(f_h)_{h\in H_{u_0}}$ where the $f_h$ are homomorphisms from $\prod\limits_{\ell=1}^n \mK_{u_\ell}$ to $\mS$. Clearly, if ${u_0}\not\in U^+$, since  $\mK_{u_0}$ is one-element, $f|_D$ is a constant map.
	
\begin{cla3} If ${u_0}\in U^+$, then any of the homomorphisms $f_h$ as above is either constant or $f_h$ depends on the coordinates  $1\leq \ell_1<\dots<\ell_m\leq n$ with $u_{\ell_1},\dots, u_{\ell_m}\in U^+$ and there exists a unique sequence of homomorphisms $\varphi_{j}\in H_{u_{\ell_j}}$, $1\leq j\leq m$, such that $$f_h(y_1,\dots,y_n)=\bigwedge_{j=1}^m y_{\ell_j}(\varphi_j).$$  \end{cla3}

 \begin{claimproof} By applying Lemma~\ref{prod_of_partial_sl} to the homomorphism $f_h\colon \prod\limits_{\ell=1}^n \mK_{u_\ell} \to \mS$, we obtain 
 $$f_h(y_1,\dots,y_n)=\bigwedge_{\ell=1}^{n} f_h^{\ell}(y_{\ell})$$ 
where $f_h^{\ell}$ is a homomorphism from $\mK_{u_\ell}$ to $\mS$ for any $1\leq \ell\leq n$. 
Suppose that $f_h$ is not constant, and let $1\leq \ell_1<\dots<\ell_m\leq n$ be the coordinates which $f_h$ depends on. Clearly, $u_{\ell_1},\dots, u_{\ell_m}\in U^+$. Then by leaving out the superfluous  meetands (the ones that correspond to coordinates which $f_h$ does not depend on) from the right side of the preceding equality we obtain
$$f_h(y_1,\dots,y_n)=\bigwedge_{j=1}^{m} f_h^{\ell_j}(y_{\ell_j}).$$
By an application of item (4) of Lemma~\ref{retract2}, for any $1\leq j\leq m$, $f_h^{\ell_j}$ is a projection to a unique coordinate $\varphi_{j}\in H_{\ell_j}$, that is, $$f_h^{\ell_j}(y_{\ell_j})=y_{\ell_j}(\varphi_j).$$ Thus,  $$f_h(y_1,\dots,y_n)=\bigwedge_{j=1}^m y_{\ell_j}(\varphi_j)$$ as claimed.~\end{claimproof}

Let $$\mG\coloneqq \dot\bigcup_{u\in U}\mS^{H_u}.$$ Clearly, $U$ is non-empty, and by item (2) of Lemma~\ref{retr}, $H_{\id}\neq \emptyset$. Observe also that by item (3) of Lemma~\ref{connected}, the $\mS^{H_u}$ are connected, so the components of $\mG$ coincide with the $\mS^{H_u}$ where $u\in U$. To complete the proof of the ``only if'' part of the first statement of the lemma, we show that $\mG$ is a compatible structure in $\var V.$

We define a subclone $\mathscr C$ of $\pol(\mG)$. An $n$-ary operation $g$ on $G$ is an element of $\mathscr C$ if and only if for every component  $\mE=\prod\limits_{\ell=1}^n \mS^{H_{u_{\ell}}}$ of $\mG^n$, there is some ${u_0}\in U$ such that $g(E)\subseteq \mS^{H_{u_0}}$, and the following properties are satisfied:

\begin{enumerate} 
\item if $H_{u_0}=\emptyset$, then $g|_E$ is the constant map to the singleton $S^{H_{u_0}}$, and
\item if ${u_0}\in U^+$, then $g|_E=(g_h)_{h\in H_{u_0}}$ where for each $h\in H_{u_0}$, $g_h\colon\mE \to \mS$, and 
\begin{itemize}
\item $g_h$ is either a constant map to $S$, or 
\item $g_h=\bigwedge_{j=1}^m p_j$ where $p_j\colon E \to S$ is a projection to some coordinate in $H_{u_{\ell_j}}$,  $1\leq \ell_1<\dots<\ell_m\leq n$, and $u_{\ell_j}\in U^+$ for each  $1\leq j\leq m$.
\end{itemize}
\end{enumerate}

 Notice that all coordinate maps (the $g_h$) in the definition of $g$ are homomorphisms from $\mE$ to $\mS$. So, $g$ is a polymorphism of $\mG$. It is also easy to see that $\mathscr C$ contains the projection operations of $G$ and is closed under composition. So,  $\mathscr C$ is indeed a subclone of $\pol(\mG)$.




\begin{cla4}
Any $n$-ary operation  $f\in\pol (\mK)$  has a unique extension in $\mathscr C$. 
\end{cla4} 

 \begin{claimproof}
First we prove that every $n$-ary operation  $f\in\pol (\mK)$ extends to an operation  $f^*\in \mathscr C$.  Let $\mD$ be an arbitrary component of $\mK^n$. So, $$\mD=\prod\limits_{\ell=1}^n \mK_{u_{\ell}}  \text{ for some }u_1,\dots,u_n\in U.$$ Let $$\mE=\prod\limits_{\ell=1}^n \mS^{H_{u_{\ell}}}$$ be the component containing $\mD$ in $\mG^n$. If $f|_ D$ is a constant map,  then we define $f^*|_E$ to be the  constant map with the same value. If $f|_D$ is not constant, then $f$ maps $\mD$ into a component $\mK_{u_0}$ of $\mK$ for some ${u_0}\in U^+$. Then $f|_D=(f_h)_{h\in H_{u_0}}$ where the $f_h$ are homomorphisms from $\prod\limits_{\ell=1}^n \mK_{u_\ell}$ to $\mS$. By applying Claim 3, for any $h$ we have that either $f_h$ is constant or is given by
$$f_h(y_1,\dots,y_n)=\bigwedge_{j=1}^m y_{\ell_j}(\varphi_j)$$
for some $u_{\ell_1},\dots, u_{\ell_m}\in U^+$ and $\varphi_{j}\in H_{u_{\ell_j}}$, $1\leq j\leq m$. If $f_h$ is a constant map,  then we define $f^*_h$ to be the  constant map with the same value.  If $f_h$ is not constant, then $f^*_h$ 
defined by $$f^*_h(z_1,\dots,z_n)=\bigwedge_{j=1}^m z_{\ell_j}(\varphi_j)\text{ where }z_{\ell}\in\mS^{H_{u_{\ell}}},\ 1\leq\ell\leq n,$$ extends $f_h$. Then $f^*|_E=(f^*_h)_{h\in H_{u_0}}$ is an extension of $f|_D$. By taking the union of the extensions $f^*|_E$ where $E$ runs through the components of $\mG^n$, we obtain  an extension $f^*\in \mathscr C$ of $f$. 

Let $g\in \mathscr C$ be any extension of $f$. Next we verify that $g$ is uniquely determined by $f$. It suffices to prove that for any components $\mD \text{ of }  \mK^n$ and $\mE\text{ of }\mG^n$ with $D\subseteq E$, $g|_E$ is a unique extension of $f|_D$.
We assume that
$$\mD= \prod\limits_{\ell=1}^n \mK_{u_\ell} \text{ and }\mE=\prod\limits_{\ell=1}^n \mS^{H_{u_{\ell}}}.$$  We also assume that for some ${u_0}\in U$, $\mK_{u_0}\subseteq \mS^{H_{u_0}}$ is the component that  $f|_D$ maps into. Notice that the product $$\mP=u_1(\mS^{\mK})\times \dots \times u_n(\mS^{\mK}) $$ is a weak substructure of $\mD.$ By Claim 2, we know that if $u_{\ell}\in U^+$, then $u_{\ell}(\mS^{\mK})$ is a substructure of $\mK_{u_\ell}$ and is isomorphic to $\mS$, and if $u_{\ell}\not\in U^+$, then $u_{\ell}(\mS^{\mK})$ and $\mK_{u_{\ell}}$ are singletons. Therefore, $\mP$ is a substructure of $\mD$ and is isomorphic to a finite power of $\mS$. 

If $f|_D$ is constant, either $P$ is one-element, so, $D$ and $E$ are one-element and $g|_E=f|_D$, or $P$ has at least two elements. In  the latter case, if ${u_0}\notin U^+$, then $K_{u_0}$ is a singleton and the uniqueness of $g|_E$ is obvious. So, we assume that ${u_0}\in U^+$. Let the coordinate maps of $f|_D$ and $g|_E$ be respectively $f_h\colon\mD \to \mS$ and $g_h\colon\mE \to \mS$ where $h\in H_{u_0}$. If $g|_E$ was not constant, then by the definition of $g$ on $E$, there would be an $h$ and 
some coordinates $1\leq \ell'_1<\dots<\ell'_{m'}\leq n$ which $g_h$ depends on such that 
$u_{\ell'_j}\in U^+$ for all  $1\leq j\leq m'$ and $g_h=\bigwedge\limits_{j=1}^{m'} p_j$ where $p_j\colon E \to S$ is a projection to some coordinate in $H_{u_{\ell'_j}}$. This is impossible, since then $g_h|_P$ would not be constant and, on the other hand, $g_h|_P=f_h|_P$ and the latter is constant. So, for all $h$, $g_h$ and so $g|_E$ must be constant. In particular,  $g|_E$ is the constant map with the same value as $f|_D$. Thus, $g|_E$ is determined by $f|_D$.

If $f|_D$ is not constant, then ${u_0}\in U^+$. Let the $g_h\colon\mE \to \mS$ be the coordinate maps of $g|_E$ where $h\in H_{u_0}$. By definition, for any $h\in H_{u_0}$, either $g_h$  is constant or $g_h$ depends on some coordinates $1\leq \ell'_1<\dots<\ell'_{m'}\leq n$ such that $u_{\ell'_j}\in U^+$ for all  $1\leq j\leq m'$ and 
$$g_h=\bigwedge\limits_{j=1}^{m'} p_j$$ where $p_j\colon E \to S$ is a projection to some coordinate $\varphi'_j$ in $H_{u_{\ell'_j}}$ for each $1\leq j\leq m'$. If $f_h$ is constant, then $g_h$ is constant with the same value by the same argument as in the preceding paragraph. If $f_h$ is not constant, then by Claim 3, there is a sequence $1\leq \ell_1<\dots<\ell_{m}\leq n$ where  $f_h$ depends only on the  coordinates $\ell_1,\dots,\ell_{m}$, and $$f_h(y_1,\dots,y_n)=\bigwedge\limits_{j=1}^{m} y_{\ell_j}(\varphi_j)$$ for all $y_{\ell}\in K_{u_{\ell}}$ where $1\leq \ell\leq n$  and for a unique sequence of $\varphi_{j}\in H_{u_{\ell_j}}$ where $1\leq j\leq m$.
Now, notice that $g_h$ depends on the same coordinates in $\{1,\dots,n\}$ as $g_h|_P$, and $f_h$ depends on the same coordinates in $\{1,\dots,n\}$ as $f_h|_P$.
This follows from the fact that $[u_{\ell}(x)](\varphi)=0$ and $[u_{\ell}(y)](\varphi)=1$ for all $u_{\ell}\in U^+$ and for all $\varphi\in H_{u_{\ell}}$ when $1\leq \ell\leq n$, that is,  the $\ell$-th factor of $P$ consists of the constant $0$ and the constant $1$ tuples. 
Since $f_h|_P=g_h|_P$, we have that $m=m'$, and the two sequences $\ell_1,\dots,\ell_m$ and $\ell'_1,\dots,\ell'_{m'}$ are equal. Thus, $g_h$ depends only on the coordinates $\ell_j$ where $1\leq j\leq m$. Now, we use that the sequence of homomorphisms $\varphi_{j}\in H_{u_{\ell_j}}$, $1\leq j\leq m$, is uniquely determined by the  $\ell_j$ for $f_h$. Since $g_h=\bigwedge\limits_{j=1}^m p_j$ where $p_j\colon E \to S$ is a projection to $\varphi'_{j}$ in $H_{u_{\ell_j}}$ and $g_h$ extends $f_h$,  $\varphi_j=\varphi'_{j}$ for each $j$. So, for any $h$, $g_h$ is unique. Thus, $g|_E$ is determined by $f|_D$.~\end{claimproof}
                  								
So, by Claim 4 and  Lemma~\ref{clonehom}, there exists a clone homomorphism from $\pol (\mK)$ to $\mathscr C$. Since $\mK$ is compatible in $\var V$, there is a clone homomorphism from the clone of $\var V$ to $\pol (\mK)$. Thus, there is a clone homomorphism from the clone of $\var V$  to $\mathscr C$ and, hence, to $\pol(\mG)$. So, $\mG$ must be compatible in  $\var V$. 

The finiteness claim of the lemma  is obvious from how we constructed  the $\mS$-powers in the above proof.
\end{proof}

In order to prove the primeness of the  Hobby-McKenzie filter in $\LL$, we require a modified version of the preceding lemma.

\begin{corollary} \label{kappa} Let $\var V$ be a variety. Then $\var V$ is a non-Hobby-McKenzie variety if and only if
for any large enough cardinal $\kappa$, the ternary structure $\dot\bigcup_{\lambda\leq \kappa}2^\kappa\mS^{\lambda}$ where $\lambda$ stands for a cardinal and $2^\kappa\mS^{\lambda}$ denotes the disjoint union of $2^\kappa$-many copies of $\mS^{\lambda}$, is a compatible structure in $\var V$. 
\end{corollary}
\begin{proof}
We prove the ``only if'' part of the claim, the other direction is immediate from the preceding lemma.
Let $\var V$ be an arbitrary non-Hobby-McKenzie variety. Then by the preceding lemma, there are some cardinals $\kappa_u,\ u\in U,$ not all zero, such that $\mD=\dot\bigcup_{u\in U}\mS^{\kappa_u}$ is a compatible relational structure in $\var V$. Let $\kappa$ be  an arbitrary infinite cardinal with $\kappa\geq\kappa_u$ for all $u\in U$ and $\kappa\geq |U|$. Note that by item (3) of Lemma~\ref{connected}, the powers $\mS^{\kappa_u}$ where $u\in U$ are the connected components of $\mD$. We determine the connected components of  the $\kappa$-th power of  $\mD$. Note that item (5) of Lemma~\ref{connected} applies with $K=\{\mS\}$, $|I|=\kappa$, and $\mH_i=\mD$ for all $i\in I$. So each component of $\mD^I$ is of the form
$$\prod_{i\in I} \mS^{\kappa_{u_i}}\cong \mS^{\sum_{i\in I} \kappa_{u_i}}$$ where the $u_i\in U$. Since $|I|=\kappa$ and $\kappa_{u_i}\leq \kappa$ for all $i\in I$, by cardinal arithmetic, $\sum_{i\in I} \kappa_{u_i}\leq \kappa$. Note that there is a component of $\mD^I$ that is isomorphic to $\mS^\kappa$ since there is a non-zero exponent $\kappa_u$ and by choosing $u_i= u$ for all $i\in I$, 
$$\mS^{\sum_{i\in I} \kappa_{u_i}}=\mS^{\kappa\kappa_{u}}=\mS^{\kappa}.$$ So $\mD_1=\mD^\kappa$ is a compatible relational structure in $\var V$ such that $\mD_1$ is a disjoint union of $\mS$-powers and the largest exponent that appears in the $\mS$-power components of $\mD_1$ is $\kappa$ and $|D_1|=2^\kappa$ also holds.

Now, we define a compatible ternary  structure $\mD_2$ of a single relation from $\mD_1$ in $\var V$ such that  $\mD_2$ is a disjoint union of $\mS$-powers where each cardinal at most $\kappa$ appears as an exponent of some $\mS$-power component. The base set of $\mD_2$ is the set of homomorphisms from $\mS$ to $\mD_1$. To define the triples of the only relation of $\mD_2$ we use  the full semilattice structure $\mY$ given in Figure~\ref{fig:gadget}. Let $r$ be a homomorphism from $\mY$ to $\mD_1$.
For any $j\in\{a,b,c\}$, let $r_j$ denote the homomorphism $r\iota_j$ where $\iota_j$ is the homomomorphism from $\mS$ to $\mY$ given by
$\iota_j(0)=d$ and $\iota_j(1)=j$. Now, we let the relation of $\mD_2$ be the set of triples $(r_a,r_b,r_c)$ where $r$ runs through the homomorphisms from $\mY$ to $\mD_1$.

\begin{figure}[H] 
\centering
\includegraphics[scale=1]{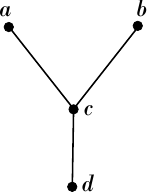}
\caption{The full semilattice $\mY$.} 
\label{fig:gadget}
\end{figure}

We note that the so defined ternary relation is indeed a relation on $D_2$, since each component of its tuples is a homomorphism from $\mS$ to $\mD_1$. It is also easy to see that $\mD_2$ is a reflexive structure. 
Since $\mD_1$ is compatible in $\var V$, there is an algebra $\alg D_1\in \var V$ such that the basic operations of  $\alg D_1$ preserve the relation of $\rel D_1$. Clearly, for any $n$-ary basic operation $t$ of $\alg D_1$ and  any homomorphisms $u_1,\dots, u_n$ from $\mS$ to $\mD_1$,  $t(u_1,\dots, u_n)$ is a homomorphism from $\mS$ to $\mD_1$. Hence, the basic operations of $\alg D_1^2$ preserve the set $D_2\subseteq D_1^2$. 

Moreover, for any $n$-ary basic operation $t$ of $\alg D_1$ and any homomorphisms $v_1,\dots, v_n$ from $\mY$ to $\mD_1$,  $t(v_1,\dots, v_n)$ is a homomorphism from $\mY$ to $\mD_1$, and for  any $j\in \{a,b,c\}$  $$t(v_1\iota_j,\dots, v_n\iota_j)=t(v_1,\dots, v_n)\iota_j.$$ So, the basic operations of $\alg D_1^2$ preserve the relation of $\mD_2$. Thus, $\mD_2$ is compatible 
in a subalgebra of $\alg D_1^2$, hence in $\var V$.

Since $\mY$ is connected, any connected component of $\mD_2$ is contained in the set of the homomorphisms from $\mS$ to $\mS^\nu$ where  $\mS^\nu$ is one of the connected components of the relational structure $\mD_1$ and $\nu\leq\kappa$. Now, we describe the substructure $\mE$ induced by the  homomorphisms from $\mS$ to $\mS^\nu$ in $\mD_2$. By looking at these homomorphisms coordinatewise, it is clear that  $\mE$ is isomorphic to the $\nu$-th power of the structure $\mE_0$ whose base set consists of  the endomorphisms of $\mS$ and whose relation consists of the triples $(r_a,r_b,r_c)\in (S^S)^3$  where $r$ is a homomorphism from $\mY$ to $\mS$. 

We represent here any endomorphism $u$ of $\mS$ by the ordered pair $(u(0), u(1))$.
So $E_0=\{(0,0),(0,1),(1,1)\}$, and the triples of $\mE_0$ are the tuples
$$((0,0),(0,0),(0,0)),\ ((0,0),(0,1),(0,0)),\ ((0,1), (0,0),(0,0)),\ ((0,1),(0,1),(0,1))$$ and $((1,1),(1,1),(1,1)).$\hfill
\medskip

Thus, $\mE_0\cong \mS\mathop{\dot\cup} \me $ where $\me$ is the ternary one-element structure  of a single non-empty relation. So, $\mE\cong (\mS\mathop{\dot\cup} \me )^\nu$. By item (5) of Lemma~\ref{connected}, up to isomorphism, the components of $\mE$ coincide with the powers $\mS^{\nu'}$ where $\nu'$ is a cardinal with $\nu'\leq \nu$. Then, by taking into account that every component of $\mD_2$ is a component of an $\mE$ related to some $\nu\leq \kappa$, every component of $\mD_2$ is isomorphic to $\mS^{\kappa'}$ for some cardinal $\kappa'\leq\kappa$. On the other hand, $\mD_1$ has a component  $\mC\cong\mS^\kappa$, and then, the components of  $\mE$ related to $\mC$ are also components of $\mD_2$. This $\mE$ has a component isomorphic to $\mS^{\kappa'}$ for any cardinal $\kappa'\leq\kappa$, hence $\mD_2$ also has such components.  Since $\mD_1$ has  $2^\kappa$-many elements,  $\mD_2$  also has  $2^\kappa$-many elements.

We saw that $\mD_2$ is compatible relational structure in $\var V$. So, there exists an algebra $\alg D_2\in \var V$ with underlying set $D_2$. Clearly, the basic operations of $\alg D_2$, no matter what they are, preserve the relation $\{(a,a,a)\colon a\in D_2\}$. Hence, the relational structure $\mL$ with base set $D_2$ and single relation $\{(a,a,a)\colon a\in D_2\}$ is a $2^\kappa$-element compatible structure in $\var V$. 

Finally, we let $\mD_3=\mD_2\times \mL$. Since $\mD_2$ and $ \mL$ are compatible in $\var V$, so is $\mD_3$. 
Clearly, by item (1) of Lemma~\ref{connected}, $\mD_3$ is a disjoint union of the powers $\mS^\lambda$ when $\lambda$ runs through the cardinals at most $\kappa$ where each copy of $\mS^\lambda$ appears precisely $2^\kappa$-many times as a component of $\mD_3$.
\end{proof}

Now, by the use of the preceding corollary, we can easily prove the main result of the paper.

\begin{theorem}
The Hobby-McKenzie filter is prime in $\LL$.
\end{theorem}
\begin{proof}
It suffices to prove that the join of any two non-Hobby-McKenzie varieties of disjoint signature is non-Hobby-McKenzie. Let $\var V$ and $\var W$ be arbitrary non-Hobby-McKenzie varieties of disjoint signatures. By the preceding corollary, there exists an infinite cardinal $\kappa$ such that the relational structure $\dot\bigcup_{\lambda\leq \kappa}2^\kappa\mS^{\lambda}$ is a compatible relational structure in both varieties $\var V$ and $\var W$. Therefore, this relational structure is also compatible in $\var V\vee\var W$. Then by using the corollary again, $\var V\vee\var W$ is a non-Hobby-McKenzie variety.
\end{proof} 
A downwardly closed sublattice of a lattice is meant by an {\em ideal}. A {\em principal ideal} of a lattice is an ideal that consists of all elements that are less than or equal to some element of the lattice. The preceding theorem just states that the non-Hobby-McKenzie types constitute an ideal in $\LL$. In the closing part of our article, we shed some light on the structure of this ideal.

Let $\mP$ be any non-trivial reflexive relational structure such that for any cardinal $\kappa$, $\mP^{\kappa}$ is connected. For example, if  $\mP$ is a finite non-trivial reflexive connected structure, then $\mP$ has this property. In particular, $\mS$ has this property. For an infinite cardinal $\kappa$, let $\var V_\kappa$ denote the variety generated by the algebra $\alg A_\kappa$ whose universe is $A_\kappa=\dot\bigcup_{\lambda\leq \kappa}2^{\kappa} P^{\lambda}$ where $\lambda$ runs through cardinals, and whose basic operations are the polymorphisms of $\mA_\kappa=\dot\bigcup_{\lambda\leq \kappa}2^{\kappa}\mP^{\lambda}$. The {\em beth numbers} are the cardinals defined by the transfinite recursion $\beth_0\coloneqq\aleph_0$ and for any non-zero ordinal $\alpha$, $\beth_{\alpha}\coloneqq\sup_{\beta<\alpha}2^{\beth_{\beta}}$ where $\beta$ runs through ordinals.  In~\cite{BGMZ}, the following proposition has been shown for a particular digraph $\mP$. However, the proof, cf. Proposition 4.5 in~\cite{BGMZ}, only uses the properties that $\mP$ is a non-trivial reflexive relational structure, and that for any cardinal $\kappa$, $\mP^{\kappa}$ is connected. So, it is really a proof of the proposition below. 

\begin{proposition} Let $\mP$ be any non-trivial reflexive relational structure such that for any cardinal $\kappa$, $\mP^{\kappa}$ is connected. For any cardinal $\kappa$, let $\var V_{\kappa}$ be the variety defined in the preceding paragraph. Then the following hold.
\begin{enumerate}
\item {For any two infinite cardinals $\kappa_1$ and $\kappa_2$ with $2^{\kappa_1}<2^{\kappa_2}$, $\var V_{\kappa_1}$ interprets in $\var V_{\kappa_2}$, but $\var V_{\kappa_2}$ does not interpret in $\var V_{\kappa_1}$.}
\item {The interpretability types of the varieties $\var V_{\beth_{\alpha}}$ where $\alpha$  is an ordinal form a chain of unbounded size in $\LL$.}
\end{enumerate}
\end{proposition}

We apply the proposition to $\mP=\mS$. Then the interpretability types of the non-Hobby-McKenzie varieties $\var V_{\beth_{\alpha}}$ 
defined via the relational structures $\mA_{\beth_{\alpha}}=\dot\bigcup_{\lambda\leq \beth_{\alpha}}2^{\beth_{\alpha}}\mS^{\lambda}$
form an unbounded chain $C$ in $\LL$. Moreover, by Corollary~\ref{kappa}, for every non-Hobby-McKenzie interpretability type, there is a member of $C$ that is greater than the given type. In other words, the ideal of the non-Hobby-McKenzie types is the ``union'' of the principal ideals determined by the members of $C$.

We finally remark that each of the relational structures $\mA_\kappa=\dot\bigcup_{\lambda\leq \kappa}2^{\kappa}\mS^{\lambda}$ has a semilattice polymorphism. So, $\mathcal S \mathcal L$  and the full idempotent reduct of the related variety $\var V_{\kappa}$ are interpretable in each other for every cardinal~$\kappa$.

\section*{Acknowledgement}
We would like to thank the anonymous referee for their detailed comments and suggestions that led to an improved version of our manuscript.


\begin{thebibliography}{}


\bibitem{BS} Bentz, W and  Sequeira, L; Taylor’s modularity conjecture holds for linear idempotent varieties. Algebra Universalis 71/2, 101–107 (2014).
\bibitem{B} Bergman, C; Universal algebra, volume 301 of Pure and Applied Mathematics (Boca Raton). CRC Press, Boca Raton, FL, 2012. 
\bibitem{Bi} Birkhoff, G; On the structure of abstract algebras, Mathematical Proceedings of the Cambridge Philosophical Society 31/4, 433-454 (1935).
\bibitem{BGMZ} Bodor, B, Gyenizse, G, Mar\'oti, M, and Z\'adori, L; Taylor is prime,  International Journal of Algebra and Computation 34/6, 857-879 (2024).
\bibitem{BS} Burris, S and Sankappanavar H. P; A course in universal algebra, volume 78 of Graduate Texts in Mathematics. Springer-Verlag, New York-Berlin, 1981.
\bibitem{C} Chicco, A; On the primeness of locally finite idempotent 3-permutability, Algebra Universalis 80/18 (2019).
\bibitem{GT} Garcia O. C and Taylor, W; The lattice of interpretability types of varieties, Mem. Amer. Math. Soc. 50, v+125, 1984.
\bibitem {GMZ1} Gyenizse, G, Mar\'oti, M, and Z\'adori, L; $n$-permutability is not join-prime for $n \geq 5$, Internat. J. Algebra Comput. 30/8, 1717 - 1737 (2020).
\bibitem {GMZ2} Gyenizse, G, Mar\'oti, M, and Z\'adori, L; Congruence permutability is prime, Proc. Amer. Math. Soc. 150/6, 2733–2739 (2022). 
\bibitem{HM} Hobby, D and McKenzie, R; The structure of finite algebras, Contemporary Mathematics 76, American Mathematical Society, Providence, RI, 1988.
\bibitem{KK} Kearnes, K. A and  Kiss, E. W; The shape of congruence lattices, Mem. Amer. Math. Soc. 222, viii+169, 2013.
\bibitem{KKVW} Kozik, M, Krokhin, A, Valeriote, M, and Willard, R; Characterizations of several Maltsev conditions, Algebra Universalis 73, 205-224 (2015).
\bibitem{KT} Kearnes, K. A and Tschantz, S. T; Automorphism groups of squares and of free algebras, Internat. J. Algebra Comput. 17/3, 461-505 (2007).
\bibitem{MMT} McKenzie, R. N, McNulty, G. F, and Taylor W. F; Algebras, lattices, varieties. Vol. I. The Wadsworth $\&$ Brooks/Cole Mathematics Series. Wadsworth $\&$ Brooks/Cole Advanced Books $\&$ Software, Monterey, CA, 1987.
\bibitem{O} Opršal, J; Taylor’s modularity conjecture and related problems for idempotent varieties, Order 35/3, 433–460 (2018).
\bibitem{S} Sequeira, L; On the congruence modularity conjecture, Algebra Universalis 55/4, 495–508 (2006).
\bibitem{VW} Valeriote, M and Willard, R; Idempotent $n$-permutable varieties, Bulletin of the London Mathematical Society 46, 870–880, (2014).
\end{thebibliography}
\end{document}